\documentclass[letterpaper,12pt]{article}
\usepackage{amsmath, mathtools, microtype}
\usepackage{amssymb}
\usepackage{amsthm}
\usepackage{graphicx}
\usepackage{verbatim}
\usepackage{color}
\usepackage{tikz}
\usetikzlibrary{shapes,arrows}

\usepackage[mathlines]{lineno}

\newtheorem{theorem}{Theorem}
\newtheorem{definition}{Definition}

\newtheorem{lemma}{Lemma}

\def\P{{\mathbb P}}     
\def\E{{\mathbb E}}     
\def\RR{\mathbb{R}}
\def\ZZ{\mathbb{Z}}

\newcommand{\tv}{\mathrm{TV}}
\newcommand{\ind}{\mathbf{1}}
\newcommand{\bSigma}{\boldsymbol{\Sigma}}

\definecolor{r}{rgb}{1,0,0}
\definecolor{b}{rgb}{0,0,0} 
\definecolor{b2}{rgb}{0,0,0} 
\definecolor{b3}{rgb}{0,0,0} 

\title{
{\color{b}
An impossibility result for phylogeny reconstruction from $k$-mer counts
}
}

\author{
Wai-Tong Louis Fan\footnote{
Department of Mathematics, Indiana University, Bloomington.} 
\footnote{Center of Mathematical Sciences and Applications, Harvard University, Cambridge.
}
\and
Brandon Legried\footnote{
Department of Statistics, University of Michigan, Ann Arbor.} 
\and 
Sebastien Roch\footnote{
Department of Mathematics, University of Wisconsin--Madison.}
}

\date{
\today
}

\begin{document}

\maketitle

\begin{abstract}
We consider phylogeny estimation under a two-state model of sequence evolution by site substitution on a tree. {\color{b}In the asymptotic regime where the sequence lengths tend to infinity, we show that for any fixed $k$ no statistically consistent phylogeny estimation is possible from $k$-mer counts over the full leaf sequences alone. Formally, we establish that the joint distribution of $k$-mer counts over the entire leaf sequences on two distinct trees have total variation distance bounded away from $1$ as the
sequence length tends to infinity. Our impossibility result implies that statistical consistency requires more sophisticated use of $k$-mer count information, such as block techniques developed in previous theoretical work.} 
\end{abstract}

\section{Introduction}

Molecular sequence comparisons 
are fundamental to many bioinformatics
methods~\cite{gusfield_algorithms_1997,durbin_biological_1998,compeau_bioinformatics_2018}. In particular, the probabilistic analysis
of sequences and their statistics 
has provided valuable insights, for instance, 
in comparative genomics~\cite{karlin_methods_1990,barbour_compound_2001,lippert_distributional_2002,reinert_alignment-free_2009}, {\color{b}
population genetics~\cite{tavare_line--descent_1984,price_principal_2006,patterson_population_2006,baik_phase_2005},} 
and phylogenetics~\cite{steel_recovering_1994,erdos_few_1999,evans_broadcasting_2000,mossel_phase_2004,roch_phase_2017}.
In this paper, we consider alignment-free phylogeny reconstruction~\cite{vinga_alignment-free_2003,haubold2014alignment}.

Alignment-free approaches are an important class of methods for estimating evolutionary trees that bypass the computationally hard multiple sequence alignment problem 
({\color{b2} depicted in Figure \ref{fig:msa}})
and avoid the need for a reference genome. Typically,
these methods construct 
pairwise distances between sequences based on match lengths \cite{ulitsky2006average, haubold2015andi} or 
$k$-mer counts \cite{qi2004whole, haubold2014alignment,fan2015assembly}. Here a $k$-mer refers to a consecutive substring of length $k$ in an input sequence ({\color{b2} see Figure \ref{fig:kmerexample} for an illustration}).
The pairwise distance matrix obtained is then used to reconstruct the phylogenetic relationships among the sequences. A variety of standard
distance-based phylogenetic methods
can be used for this purpose~\cite{Warnow:17,Steel:16}.
Numerous popular pipelines are available that implement these alignment-free approaches~\cite{haubold2015andi,ondov2016mash,lees2018evaluation,lees2019fast},
although they do not offer rigorous guarantees of accurate reconstruction.
\begin{figure}
    \centering
    \includegraphics[scale=0.2]{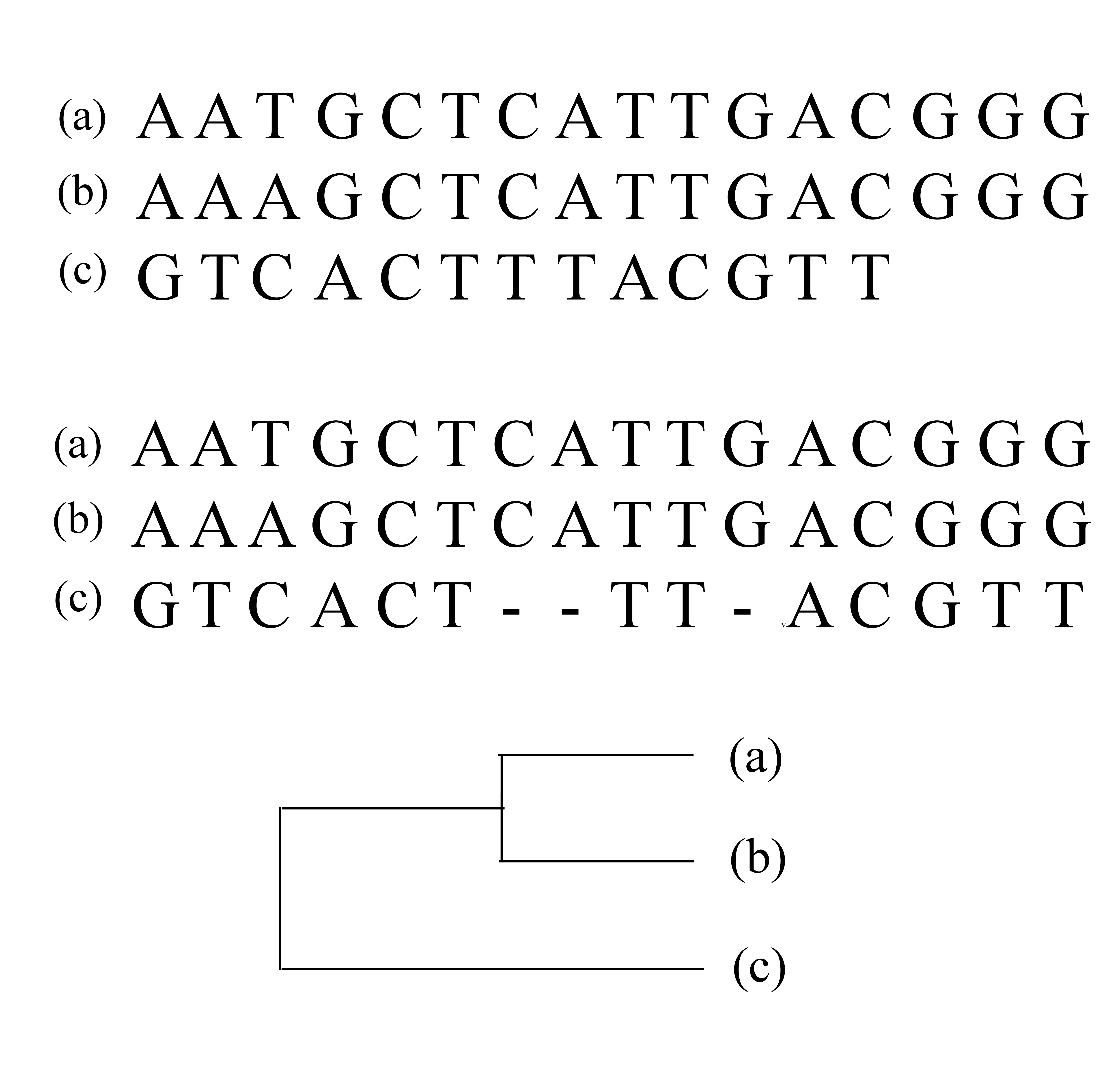}
    \caption{{\color{b3} Standard steps in phylogeny estimation.
    Top:  DNA sequences obtained from species (a), (b), and (c). While inherited from a common ancestor, the sequences and their lengths differ because of past mutations (including insertions and deletions). Middle:  A multiple sequence alignment of the sequences, where gaps are inserted to align the columns as best as possible.  Each column indicates inferred common ancestry (homology).  Bottom:  A rooted phylogenetic tree depicting the estimated evolutionary history of the sequences, with (a) and (b) being more closely related.}}
    \label{fig:msa}
\end{figure}
\begin{figure}
    \centering
    \includegraphics[scale=0.2]{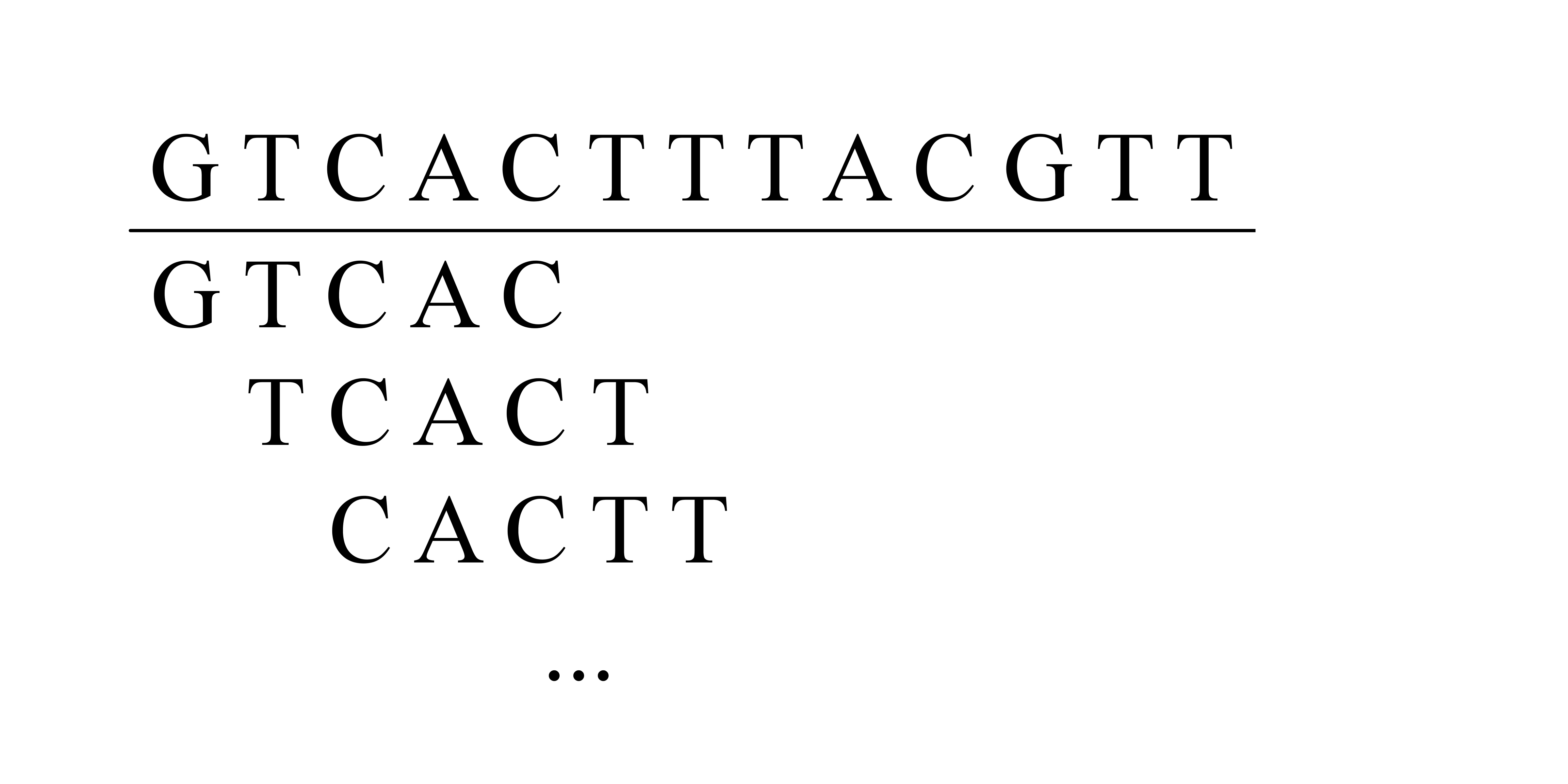}
    \caption{
    {\color{b3} For a given sequence, the $k$-mer counts are obtained by reading words of length $k$ starting from each site and then counting how many times each possible length-$k$ word appears.  Here $k = 5$.}
    }
    \label{fig:kmerexample}
\end{figure}

In this paper, we consider the problem of phylogeny estimation under a two-state symmetric model of sequence evolution by site substitutions on a leaf-labeled tree. 
In the asymptotic regime where the sequence length tends to infinity, we show that:
\begin{quote}
\begin{center}
\emph{for any fixed $k$, no statistically consistent phylogeny estimation is possible from the $k$-mer counts of the {\color{b}entire} input sequences alone.} 
\end{center}
\end{quote}
{\color{b}Formally, we establish that the joint distribution of $k$-mer counts over the entire leaf sequences on two distinct trees have total
variation distance bounded away from 
$1$ as the sequence length tends to infinity. 
Put differently, these two joint distributions 
have a non-vanishing overlap in that asymptotic regime.
Our results are information-theoretic:
since the reconstruction probability of any method is only as good as the worst total variation distance (see~\cite[Lemma 3.2]{fan2018}), our main claim (Theorem \ref{T:main}) implies
an impossibility result
for reconstruction methods using
only $k$-mer counts across the entire sequences at the leaves. On the other hand, our results have no implications for reconstruction methods using $k$-mer counts in more elaborate ways, e.g., through block decomposition. We come back to prior approaches of this type below in ``Related work.''}

To bound the total variation distance between the two distributions on well-chosen trees, 
our proof 
takes advantage of a multivariate local central limit theorem, an approach which is complicated by the
probabilistic and {\color{b} linear} dependencies
of $k$-mers. 

\paragraph{Related work}
A related impossibility result was established in our previous work \cite{fan2020impossibility}, where it was shown that no consistent distance estimation is possible from \emph{sequence lengths alone} under the TKF91 model~\cite{Thorne1991}, a more complex model of sequence evolution
which also allows for insertions and
deletions (indels). On the other hand, 
sequence lengths are significantly
simpler to analyze than $k$-mers
and are not used in practice to 
infer phylogenies. Moreover the
results in~\cite{fan2020impossibility}
only apply to distance-based 
phylogeny estimation methods, while
our current results are more general.

In a separate line of work, 
a computationally efficient 
algorithm for alignment-free phylogenetic reconstruction 
was developed and analyzed in 
\cite{daskalakis2013alignment}.
Rigorous sequence length 
guarantees for high-probability
reconstruction under an indel model
related to the TKF91 model
were established. 
While this method is based on $1$-mers, it first divides up the
input sequences into blocks
of an appropriately chosen length
and then compares the $1$-mer counts
on each block across sequences. {\color{b2} A block in the ancestral sequence gives rise to blocks in the descendant sequences that have 
the same 
position, 
so the comparison does not require an alignment technique.}
The weak correlation between the blocks
allows the use of concentration
inequalities on a notion of pairwise
distance proposed in \cite{daskalakis2013alignment}. {\color{b}In 
particular, this 
reconstruction method uses more information
than $1$-mer counts over the entire sequences  (that is, it uses $1$-mer counts \textit{over each separate block}), so that
our results do not apply to it (in the limit of zero indel rate). However 
no practical implementation {\color{b2}of this $1$-mer-based approach} is available.
In \cite{daskalakis2013alignment}, 
accurate phylogenetic reconstruction with high probability is shown to be achievable when
the sequence length is of polynomial order in the number of leaves $n$. In a constrained regime of parameters, significantly improved bounds on the sequence length requirement 
were obtained 
 in~\cite{ganesh_optimal_2019}, who used techniques related to those of \cite{daskalakis2013alignment} and also
exploited a well-studied connection to ancestral sequence
reconstruction.
In other related work, it was proved in~\cite{Allman2015StatisticallyCK}
that the tree topology as well as
mutation parameters can be identified from pairwise joint $k$-mer count distributions under more general
substitution-only models of sequence evolution using an appropriately
defined notion of distance. Block techniques can then be used to derive statistical consistency results. 
See also~\cite{durden_identifiability_2019} for extensions to coalescent-based
models. We emphasize that the results in \cite{daskalakis2013alignment, Allman2015StatisticallyCK, ganesh_optimal_2019, durden_identifiability_2019} do not contradict our main claim (Theorem \ref{T:main}), which excludes block decomposition.
}



Alignment-free sequence comparisons
based on $k$-mer counts were also
studied for independent sequences
with i.i.d.~sites or under certain hidden Markov models of sequences~\cite{lippert_distributional_2002,reinert_alignment-free_2009,wan_alignment-free_2010,barbour_compound_2001}.
Because they assume independent sequences, such results are not directly relevant to phylogeny reconstruction.



\paragraph{Organization} 
The paper is organized as follows. 
In Section~\ref{S:Result}, we state our main results after providing the necessary background and definitions. We also sketch the main steps of the proof.
The details of the proof can be found 
in Section~\ref{S:proof}. A few auxiliary results are in the appendix.

\section{Definitions and main result}\label{S:Result}

In this section, we state our result formally, after introducing the relevant concepts.

\paragraph{$k$-mers.} Let $k$ be a positive integer, fixed throughout.
First, we define $k$-mers and introduce their frequencies in a binary sequence, which will serve as  our main statistic.
\begin{definition}\label{Def:kmer}
A \textbf{$k$-mer} is a {\color{b}binary string} of length $k$, i.e., $y\in \{0,1\}^{k}$.
For a binary sequence $\sigma=(\sigma_i)_{i=1}^m$ of length $m$, we let
$f_{\sigma}(y) \in \ZZ_+$ be  the number of times $y$ appears in $\sigma$ as a consecutive substring,
{\color{b}where $\ZZ_+$ is the set of non-negative integers. That is,}
\[
f_{\sigma}(y)
=
\sum_{i=0}^{m-k}\ind\{(\sigma_{i+1},\ldots, \sigma_{i+k})= y\}.
\]
The \textbf{frequency vector (or count vector) of $k$-mers in $\sigma$} is the vector 
$$f_{\sigma}=(f_{\sigma}(y))_{y\in \{0,1\}^{k}}\in \ZZ_+^{2^k}.$$ 
The coordinates of $f_{\sigma}$ are ordered 
 such  that the $j$-th coordinate is the frequency of the $k$-mer that is the {\color{b} base-2 numeral representation of $j-1$.}
\end{definition}

For example,
when $k=1$, the count vector of $1$-mers of a binary sequence is $(a,b)$ where $a$ is the number of zeros and $b$ is the number of ones. 
Hence,  the count vector of $1$-mers of $00111000$ is $(5,3)$.
When $k=2$, there are $2^k=4$ binary strings, namely 
$\{(00), (01),(10),(11)\}$. So the count vector of $2$-mers of the sequence $001111000$ is $(3,1,1,3)$ since $(00)$ appears 3 times, $(01)$ appears one time, etc. 
By convention, the count vector of $k$-mers for any binary sequence with length less than $k$ is equal to $(0,\cdots,0)\in \ZZ^{2^k}_{+}$.


\paragraph{Probabilistic model of sequence evolution.}
We consider a symmetric substitution model on phylogenies, also known as the {\color{b} Cavendar-Ferris-Neyman (CFN) model~\cite{farris_probability_1973,cavender_taxonomy_1978,Neyman:71},} for binary sequences of fixed length $m$. The CFN model on a single edge of a {\color{b}metric tree} is a continuous-time Markov process with state space $\{0,1\}^m$ such that (i) the $m$ digits are independent and (ii) each of the $m$ digits follows a continuous-time  Markov process with two states $\{0,1\}$ that switches state at rate {\color{b}$1$}. 

We are interested in this process on a rooted {\color{b}metric tree $T$}, i.e., indexed by all points {\color{b2} along the edges of} $T$.  
{\color{b} We view an edge of length $\ell$ as the interval $[0,\ell]$ for the continuous-time substitution process.}
The root vertex $\rho$ is assigned a state $\mathcal{X}_{\rho} \in \{0,1\}^m$, drawn from the uniform distribution on $\{0,1\}^m$.  This state then evolves down the tree (away from the root) according to the following recursive process.  Moving away from the root, along each edge $e = (u,v)$ starting at $u$, we run the CFN process for a time $\ell_{(u,v)}$ with initial state $\mathcal{X}_{u}$,  described in the previous paragraph. Such processes along different edges starting at $u$ are conditionally independent, given $\mathcal{X}_{u}$.
Denote by $\mathcal{X}_{t}$ the resulting state at $t \in e$.  
Then the full process,  denoted by $\{\mathcal{X}_{t}\}_{t \in T}$, is called the \textbf{CFN model on tree $T$}.  
In particular, the set of leaf states is $\mathcal{X}_{\partial T} = \{\mathcal{X}_{v}:v \in \partial T\}$.
It is clear that, under this process, the $m$ digits remain independent.
For more background on the CFN model, see e.g.~\cite{Steel:16}.  

\paragraph{An impossibility result.} Our main result is the following. Recall that the total variation distance between two probability measures
$\nu_1$ and $\nu_2$ on a countable space $E$ is defined by
\begin{equation}\label{Def_TV}
    \|\nu_1 - \nu_2\|_\tv
=\sup_{{\color{b} E'} \subseteq E} \left|
\nu_1({\color{b} E'} ) -\nu_2({\color{b} E'} ) \right|.
\end{equation}

\begin{theorem}\label{T:main}
Fix $k\in\mathbb{N}$.
For any $n \geq 3$, there exists distinct trees $T_1 \neq T_2$ with $n$
leaves such that
\begin{equation}\label{main}
{\color{b}\limsup_{m \to\infty}}\,
\|\mathcal{L}^{(1)}_m - \mathcal{L}^{(2)}_m\|_\tv \,<\, 1,
\end{equation}
where 
$\mathcal{L}^{(i)}_m$ is the law of the $k$-mer frequencies of the leaf sequences of length $m$ under the CFN model on tree $T_i$. Furthermore, the trees $\{T_1,\,T_2\}$ can be chosen to be independent of $k$.
\end{theorem}





From \eqref{Def_TV}, we see that \eqref{main} implies the following: using only the $k$-mer frequencies over the {\color{b}entire} leaf sequences for a fixed $k\geq 1$, there is no statistical test that can distinguish between $T_1$ and $T_2$ with {\color{b} success} probability going to 1 as the sequence length tends to $+\infty$. More precisely, by  \eqref{main} and the reconstruction upper bound in part 1 of {\color{b} \cite[Lemma 3.2]{fan2018},} there exists $\epsilon>0$ such that the probability that a tree estimator gives the correct estimate is at most $1-\epsilon$, uniformly for all estimators and all integers $m\geq k$. {\color{b} We point out again that our results do not apply to block decomposition methods.}


\medskip

\noindent
{\bf Proof sketch. }
Since our goal is to prove a negative result, we get to pick the trees. We consider two trees $\{T_1,\,T_2\}$ that 
have the same set of $n$ leaves and are the same except for the {\color{b}placement of a single edge leading to} leaf $A$. These trees are depicted in {\color{b}Figure \ref{fig:reduction}} and described in detail in Section \ref{S:Trees} below. 
The topologies of $\{T_1,T_2\}$ differ only on the subtree containing three leaves $\{A,B,C\}$ that have the same distance from the root.

We seek to distinguish the law of the $k$-mer frequencies of the $n$ leaf sequences between the two trees. This will be done in two steps, in Sections \ref{S:Reduction} and \ref{S:LocalCLT} respectively, and concluded in  Section \ref{S:Final}.
\begin{enumerate}
    \item \textit{Step 1 (Reductions):} By using the Markov property of the CFN process on trees, we first reduce the problem from $n$ leaf sequences to only  $3$ sequences (Lemma \ref{lem:red-1}).  We can assume the sequence length $m$ is a multiple of $k$ (Lemma \ref{L:reduc_hatf}). Then $k$-mer frequencies are  functions of {\color{b}pairs of adjacent, non-overlapping $k$-mers}, together with the first and the last $k$-mers (Lemma \ref{lem:new-space}). {\color{b} For short, we refer to these pairs as ``adjacent $k$-mer pairs'' and a precise definition is in \eqref{Def:NonOverlap} below.}  We can further reduce the problem to distinguishing the {\color{b} laws} of {\color{b}adjacent} $k$-mer pairs (Lemma \ref{lem:red-2}). The {\color{b} collection of adjacent} $k$-mer pairs satisfy certain linear relations (Lemma \ref{lem:flowconstraints}), which lead to  redundancy that we need to address (Lemma \ref{lem:redundant}). Summarizing, the problem is reduced to distinguishing  the {\color{b} laws} of  non-redundant, {\color{b}adjacent} $k$-mer pairs on three points $\{A,\,B',\,C'\}$ as depicted in Figure \ref{fig:three} (Lemma \ref{lem:red-3}).
    
    \item \textit{Step 2 (Applying a local CLT):} We apply a local central limit theorem for i.i.d.~vectors to the law of non-redundant, {\color{b}adjacent} $k$-mer pairs as $m\to\infty$ (Lemmas~\ref{L:minfinity} and \ref{L:minfinity2} and Theorem \ref{T:DMCLT}). 
    Non-redundancy guarantees the non-degeneracy of the limit distribution (Lemmas~\ref{lem:pd}, \ref{lem:davis}, \ref{lem:var}  and \ref{lem:expec}).
    The two limit normal distributions, under the two trees respectively, have an overlap (Lemmas \ref{lem:local-clt} and \ref{L:sizeY}) and therefore {\color{b}so do} the laws of non-redundant, {\color{b}adjacent} $k$-mer pairs.
\end{enumerate}

Section \ref{S:Final} concludes the proof of  Theorem \ref{T:main}.

\section{Proof}
\label{S:proof}

In this section, we give the details of the proof. Some standard results are stated in the appendix.

\subsection{The two trees}\label{S:Trees}
We {\color{b}consider two rooted metric trees  $\{T_1,\,T_2\}$ as follows.}
\begin{enumerate}
    \item  $T_1$ and $T_2$ have the same set of $n\geq 3$ leaves $\{A,B,C,X^{4},...,X^{n}\}$.

    \item The subtree of $T_1$ restricted to the $n-1$ leaves $\{B,C,X^{4},...,X^{n}\}$ is the same as that of $T_2$. {\color{b}Here the restriction of a metric tree $T$ to a subset $L$ of leaves is the metric tree obtained from $T$ by keeping only those points lying on a path between two leaves in $L$.}

    \item The subtrees of $T_1$ and $T_2$ below the most recent common ancestor (MRCA) of $\{A, B, C\}$ contain none of $\{X^{4},...,X^{n}\}$.
    
    \item  Leaves $\{A,B,C\}$ satisfy, for $i\in\{1,2\}$,
    \begin{equation}\label{topo_ABC} 
    \begin{split}
        \mathrm{dist}_{T_i}(\rho,C)=&\, \mathrm{dist}_{T_i}(\rho,B) \quad\text{and}\\
            \mathrm{dist}_{T_1}(A,C)=&\, \mathrm{dist}_{T_2}(A,B) < \mathrm{dist}_{T_i}(B,C),
    \end{split}
    \end{equation} 
    where 
    {\color{b}
    $\mathrm{dist}_{T_i}(x,y)$ denotes the sum of edge lengths along the path from $x$ to $y$ in the tree $T_i$.
    }
\end{enumerate}

These trees are depicted in {\color{b}Figure \ref{fig:reduction}}, 
where
$X = (X^{4},...,X^{n})$ refers to the set of all leaves other than $\{A,B,C\}$.  In \textit{Newick tree format} (see, e.g., \cite{Warnow:17}), the topology of $T_1$ restricted to $\{A,B,C\}$ is $((A,C),B)$, while the topology of $T_2$ restricted to $\{A,B,C\}$ is $((A,B),C)$.
Clearly, $\{T_1,T_2\}$ does not depend on $k$, and their topologies differ only on the subtree containing three leaves $\{A,B,C\}$. The topology of the trees restricted to $X$ is arbitrary and plays no role in the argument.


\medskip

\noindent
{\bf Notation. }
For $i\in\{1,2\}$, we let $\P^{(i)}=\P^{(i),m}$ be the probability measure of the CFN model on $T_i$  with sequence length $m$ (recall that the root state is drawn from the uniform distribution on $\{0,1\}^m$),
and $\mathbb{P}^{(i)}_{\Theta}$ be the law of a random variable $\Theta$ under $\P^{(i)}$.
For a binary sequence $\sigma_Z=\sigma(Z)$ at a point $Z\in T_i$ where $i\in\{1,2\}$, we let $f_Z:=f_{\sigma(Z)}\in \ZZ_+^{2^k}$ be the $k$-mer count vector in $\sigma(Z)$ (see Definition \ref{Def:kmer}). For a finite ordered set of points $U=(u_j)$ on the tree $T_i$, we let $f_U=(f_{u_j}) \in \ZZ_+^{2^k\times |U|}$. With this notation, in Theorem \ref{T:main}, $\mathcal{L}^{(i)}_m=\P^{(i)}_{f_{X},f_{A},f_{B},f_{C}}$. We also write $a\wedge b=\min\{a,b\}$ and $a\vee b=\max\{a,b\}$.

\subsection{Reductions}\label{S:Reduction}

Our argument proceeds through a series of reductions.

\subsubsection{Reduction to three vertices}

First we shall reduce the complexity of the problem from $n$ to three vertices. For this we define two internal vertices $B'$ and $C'$ on both $T_1$ and $T_2$ as follows.  
Let $C'$ be the  MRCA of $A$ and $C$ on $T_1$, and label $C'$ as well {\color{b} the point
on the path between $C$ and $B$ on $T_2$ such that $\mathrm{dist}_{T_1}(C,C')=\mathrm{dist}_{T_2}(C,C')$.}
Similarly, we let 
$B'$ be the MRCA of $A$ and $B$ on $T_2$, and label $B'$ as well {\color{b} the point
on the path between $A$ and $B$ on $T_1$ such that $\mathrm{dist}_{T_1}(B,B')=\mathrm{dist}_{T_2}(B,B')$.
} This setup is depicted in Figure \ref{fig:reduction}.



{\color{b}
Our first reduction lemma asserts that we can reduce the problem 
to one of distinguishing between the 
two three-vertex trees depicted in Figure~\ref{fig:three}.
}

\begin{figure}
    \centering
    \includegraphics[width=5in]{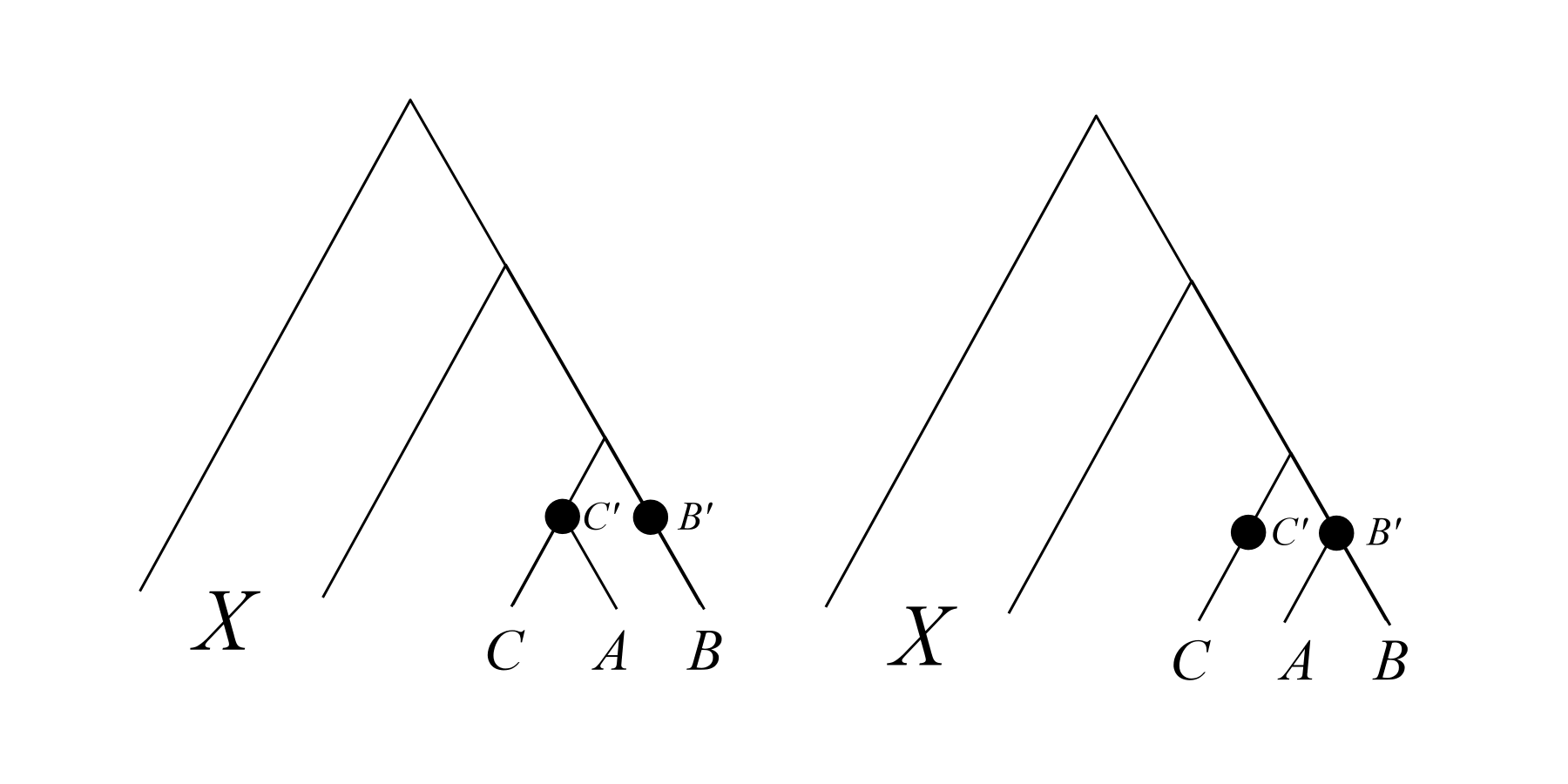}
    \caption{The trees $T_1$ (left) and $T_2$ (right) on $n$ leaves with points $C'$ and $B'$ added. {\color{b}Here $X$ refers to the remaining $n-3$ leaves.} 
    } 
    \label{fig:reduction}
\end{figure}
\begin{lemma}[Reduction to $3$ vertices]
\label{lem:red-1}
Let $T_1$ and $T_2$ be the trees with points $C'$ and $B'$ as described above.  Then for all $m\in\mathbb{N}$,
\begin{align*}
\left\|\mathcal{L}^{(1)}_m - \mathcal{L}^{(2)}_m\right\|_\tv
\leq 
\left\|\P^{(1)}_{f_{A},f_{B'},f_{C'}} - \P^{(2)}_{f_{A},f_{B'},f_{C'}}\right\|_\tv.
\end{align*}
\end{lemma}
\begin{proof} First, $(f_{X},f_{A},f_{B},f_{C})$ is of course a function of $(f_{X},f_{A},f_{B},f_{C},f_{B'},f_{C'})$, so Lemma~\ref{lem:tv-f} in the appendix implies
\begin{align*}
    \left\|\mathcal{L}^{(1)}_m - \mathcal{L}^{(2)}_m\right\|_{\tv} &= \left\|\P^{(1)}_{f_{X},f_{A},f_{B},f_{C}} - \P^{(2)}_{f_{X},f_{A},f_{B},f_{C}}\right\|_{\tv} \\
    &\leq \left\|\P^{(1)}_{f_{X},f_{A},f_{B},f_{C},f_{B'},f_{C'}} - \P^{(2)}_{f_{X},f_{A},f_{B},f_{C},f_{B'},f_{C'}}\right\|_{\tv}.
\end{align*} 
Also $f_{X,C,B} \to f_{B'C'} \to f_A$ forms a Markov chain under both $\P^{(1)}$ and $\P^{(2)}$, satisfying all the conditions of Lemma~\ref{lem:tv-markov} in the appendix.  Hence \begin{align*}
    \left\|\P^{(1)}_{f_{X},f_{A},f_{B},f_{C},f_{B'},f_{C'}} - \P^{(2)}_{f_{X},f_{A},f_{B},f_{C},f_{B'},f_{C'}}\right\|_{\tv} = \left\|\P^{(1)}_{f_{A},f_{B'},f_{C'}} - \P^{(2)}_{f_{A},f_{B'},f_{C'}}\right\|_{\tv},
\end{align*} giving the result.
\end{proof}

\begin{figure}
    \centering
    \includegraphics[width=4in]{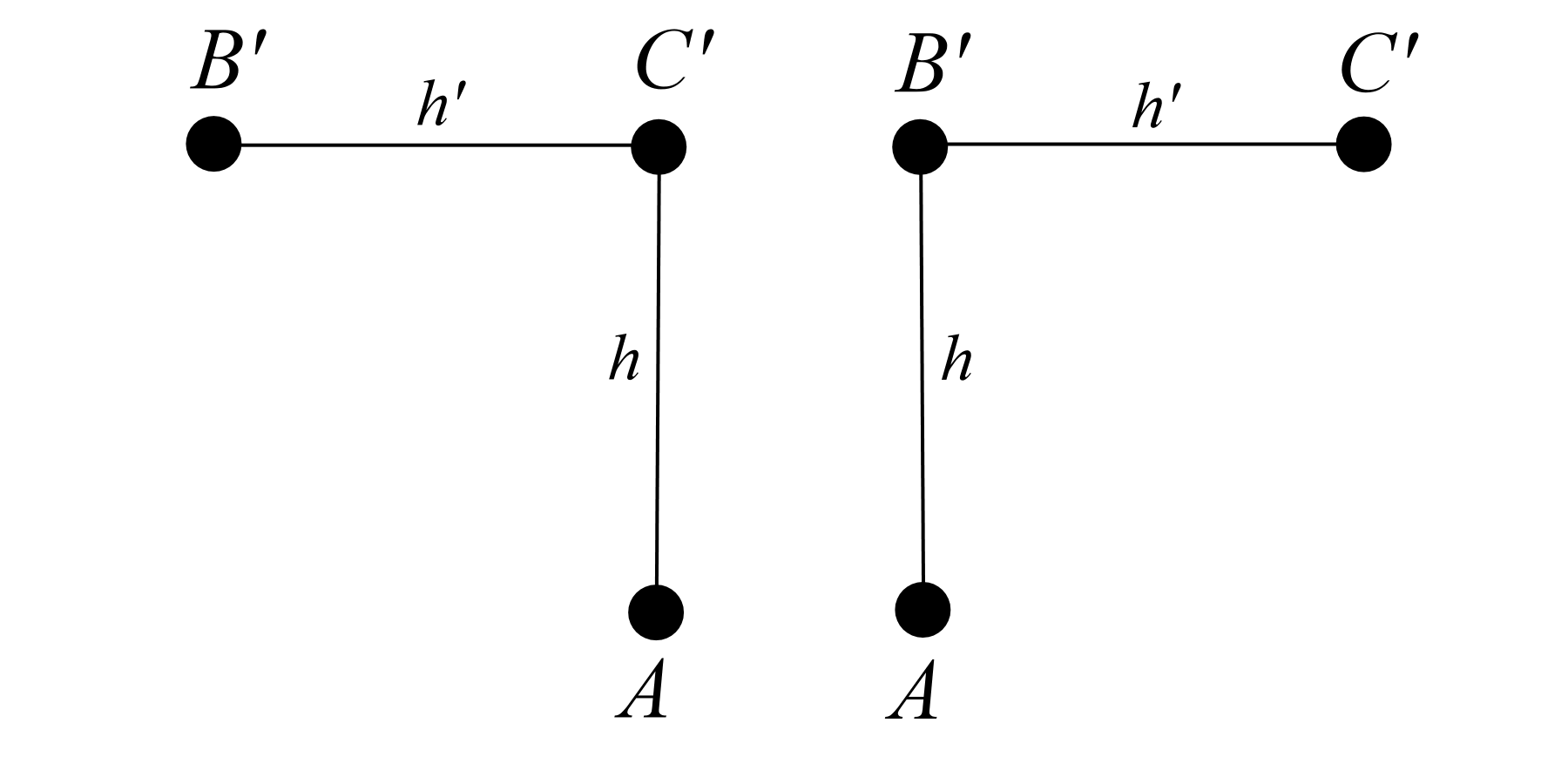}
    \caption{The three-vertex configurations for the measures {\color{b}$\P^{(1)}_{f_{A},f_{B'},f_{C'}}$ (left) and $\P^{(2)}_{f_{A},f_{B'},f_{C'}}$ (right) respectively. {\color{b2} In this figure, $h'=\mathrm{dist}_{T_1}(B',C') = \mathrm{dist}_{T_2}(B',C')$ and $h=\mathrm{dist}_{T_1}(A,C')=\mathrm{dist}_{T_2}(A,B')$.}}
    }
    \label{fig:three}
\end{figure}


\subsubsection{Reduction to {\color{b}transitions between adjacent, non-overlapping $k$-mers}}

Due to the following lemma, we can assume $m=(\mu+1)k$ for some $\mu\in\mathbb{N}$.

\begin{lemma}[Reduction to multiples of $k$]\label{L:reduc_hatf}
If $\bar{\mu}k<m<(\bar{\mu}+1)k$ where $\bar{\mu}\in\mathbb{N}$, then
\begin{align*}
\left\|\P^{(1),m}_{f_A,f_{B'},f_{C'}} - \P^{(2),m}_{f_{A},f_{B'},f_{C'}}\right\|_\tv
\leq 
\left\|\P^{(1),\,(\bar{\mu}+1)k}_{\widehat{f}_A,\widehat{f}_{B'},\widehat{f}_{C'}} - \P^{(2),\,(\bar{\mu}+1)k}_{\widehat{f}_A,\widehat{f}_{B'},\widehat{f}_{C'}}\right\|_\tv,
\end{align*}
where $\widehat{f}_V=(f_V,\,\sigma_V^{last})$ is the $k$-mer count vector together with the last $2k$-digits  $\sigma_V^{last}$ in $\sigma_V$.
\end{lemma}

\begin{proof}
Note that all digits are independent under the CFN model and
\begin{align*}
\left\|\P^{(1),m}_{f_A,f_{B'},f_{C'}} - \P^{(2),m}_{f_{A},f_{B'},f_{C'}}\right\|_\tv
=\left\|\P^{(1),\,(\bar{\mu}+1)k}_{f^m_A,f^m_{B'},f^m_{C'}} - \P^{(2),\,(\bar{\mu}+1)k}_{f^m_A,f^m_{B'},f^m_{C'}}\right\|_\tv,
\end{align*}
where $f^m_V$ is the $k$-mer count vector of the first $m$ digits of $\sigma_V$. 
The proof is complete by
Lemma \ref{lem:tv-f} since $f^m_V$ is a function of $f_V$ and the last $2k$-digits 
 when $\sigma_V$ has length $(\bar{\mu}+1)k$.
\end{proof}




For $\sigma\in\{0,1\}^m$ where $m=(\mu+1)k$, we
let $x_0^\sigma, \ldots, x_{\mu}^\sigma \in \{0,1\}^{k}$ be the  \textit{{\color{b}adjacent, non-overlapping}} $k$-mers in $\sigma$. That is, 
\begin{equation}\label{Def:NonOverlap}
\sigma=\underbrace{(\sigma_1,\ldots,\sigma_k)}_{x_0^\sigma}\,\underbrace{(\sigma_{k+1},\ldots,\sigma_{2k})}_{x_1^\sigma}\,\cdots \,\underbrace{(\sigma_{\mu k+1},\ldots,\sigma_{(\mu+1) k})}_{x_{\mu}^\sigma} \,\in \{0,1\}^{(\mu+1)k}.
\end{equation}
For $y, z \in \{0,1\}^{k}$, let $N_{y,z}^\sigma$ be the number of {\color{b}adjacent} $(y, z)$ pairs in this representation of $\sigma$: 
\begin{equation}\label{Def:N_yz}
N_{y,z}^\sigma 
=\sum_{j=0}^{\mu-1}\ind\{x_j^\sigma = y,\,x_{j+1}^\sigma=z\}.
\end{equation}
We call $N_{y,z}^\sigma$ the number of {\color{b}adjacent} transitions from $y$ to $z$.



The following lemma and its proof give an expression for $k$-mer frequencies in terms of the numbers of  {\color{b}adjacent} $k$-mer pairs as well as the ending $k$-mers.  




\begin{lemma}[$k$-mers as a function of {\color{b}adjacent} transitions]
\label{lem:new-space}
For any $\sigma \in \{0,1\}^{(\mu+1)k}$ and $\mu\in\mathbb{N}$,
the frequency vector $f_{\sigma}$ is a function of 
$$
\big(x_\mu^\sigma,  (N_{y,z}^\sigma)_{y,z \in \{0,1\}^{k}} \big).
$$
\end{lemma}
\begin{proof}
We split the set $\{0,1,\ldots, \mu k\}$ into the disjoint union $\left(\cup_{a=0}^{k-1}\Lambda_a\right) \cup \{\mu k\}$, where $\Lambda_a=\{a, k+a, 2k+a, \cdots, (\mu-1)k+a\}$ contains $\mu$ integers with remainder $a$ when divided by $k$. By definition, for  $w=(w_1,\ldots, w_k)\in \{0,1\}^{k}$,
\begin{align}
f_{\sigma}(w)
&= \sum_{i=0}^{\mu k}\ind\{(\sigma_{i+1},\ldots, \sigma_{i+k})= w\}\nonumber\\
&= \ind\{x_{\mu}^\sigma= w\} +\sum_{a=0}^{k-1}\sum_{i\in \Lambda_a}\ind\{(\sigma_{i+1},\ldots, \sigma_{i+k})= w\}.\label{fw}
\end{align}

For $a=0$, the set $\Lambda_0$ coincides with the multiples of $k$ from $0$ up to $\mu - 1$.  So
\begin{equation}\label{a0}
\sum_{i\in \Lambda_0} \ind\{(\sigma_{i+1},\ldots, \sigma_{i+k})= w\}=
\sum_{i=0}^{\mu-1}  \ind\{x_i^\sigma= w\}=\sum_{z\in \{0,1\}^{k}}N_{w,z}^\sigma.
\end{equation}
For $a\in \{1,\cdots ,k-1\}$, 
\begin{equation}\label{a ge 1}
\sum_{i\in \Lambda_a}\ind\{(\sigma_{i+1},\ldots, \sigma_{i+k})= w\}=\sum_{(y,z)\in  \Theta_a(w)}N_{y,z}^\sigma,
\end{equation}
where $\Theta_a(w)$ is the set of all pairs of the form
\begin{align*}
\big((\theta_0, \cdots ,\theta_{a-1}, \,w_1,\cdots,w_{k-a}),\,(w_{k-a+1},\cdots, w_k,\, \theta_a,\cdots,\theta_{k-1})\big) \in \{0,1\}^{2k},
\end{align*}
where $(\theta_0,\cdots,\theta_{k-1})$ is an arbitrary element in $\{0,1\}^{k}$. 

The result then follows when we put \eqref{a0} and \eqref{a ge 1} into \eqref{fw}, seeing that $f_{\sigma}(w)$ depends {\color{b}only} on the specified value of $\left(x_{\mu}^{\sigma},\left(N_{y,z}^{\sigma}\right)_{y,z \in \{0,1\}^{k}}\right)$.  This completes the proof of the lemma. \end{proof}

 
For points $V\in\{A, B',C'\}$ on the trees we let
\[Z_V = \Big( (x_0^{\sigma_V},x_1^{\sigma_V}),\, (x_{\mu-1}^{\sigma_V}, x_\mu^{\sigma_V}),\,  \big(N_{y,z}^{\sigma_V})_{y,z \in \{0,1\}^{k}} \Big),\] 
where $\sigma_V$ is the binary sequence at $V$. Note that we included 
$x_0^{\sigma_V}$, 
$x_1^{\sigma_V}$ and $x_{\mu-1}^{\sigma_V}$ here
for reasons that will become clear below.  
\begin{lemma}[Reduction to {\color{b}adjacent} transitions]
\label{lem:red-2}
Suppose $m=(\mu+1)k$ for some $\mu\in\mathbb{N}$. Then 
\begin{align*}
\left\|\P^{(1)}_{\widehat{f}_A,\widehat{f}_{B'},\widehat{f}_{C'}} - \P^{(2)}_{\widehat{f}_{A},\widehat{f}_{B'},\widehat{f}_{C'}}\right\|_\tv
\leq 
\left\|\P^{(1)}_{Z_A, Z_{B'}, Z_{C'}} - \P^{(2)}_{Z_A, Z_{B'}, Z_{C'}}\right\|_\tv.
\end{align*}
\end{lemma}
\begin{proof}
Recall the definition of $\widehat{f}_V=(f_V,\,\sigma_V^{last})$ in Lemma \ref{L:reduc_hatf}.
By Lemma \ref{lem:new-space}, $(\widehat{f}_{A},\widehat{f}_{B'},\widehat{f}_{C'})$ is a function of $(Z_{A},Z_{B'},Z_{C'})$.  Lemma \ref{lem:tv-f} gives the result.
\end{proof}

\subsubsection{Dealing with redundancy}

The quantities $\{N_{y,z}^{\sigma}\}_{y,z \in \{0,1\}^{k}}$ satisfy certain linear relations described in Lemma \ref{lem:flowconstraints} below. We will get rid of these redundancies in Lemma \ref{lem:redundant}, which will be needed for a non-degeneracy condition in the local CLT; see Lemma \ref{lem:davis} below.

\begin{lemma}[Combinatorial constraints]
\label{lem:flowconstraints}
For any $\sigma \in \{0,1\}^{(\mu+1)k}$, $\mu\in\mathbb{N}$ and $z\in \{0,1\}^{k}$,
\begin{align}
\label{eq:nflow}
\mathbf{1}\{x_0^{\sigma} = z\}+\sum_{y\in \{0,1\}^{k}:\,y \neq z} N_{y,z}^{\sigma}
\;= \; \mathbf{1}\{x_\mu^{\sigma} = z\}+\sum_{y'\in \{0,1\}^{k}:\,y' \neq z} N_{z,y'}^{\sigma}.
\end{align}
Moreover 
\begin{align}
\label{eq:nsum}
\sum_{y,z \in \{0,1\}^{k}} N_{y,z}^\sigma = \mu.
\end{align}
\end{lemma}

\begin{proof}
Equation \eqref{eq:nsum} holds since the total number of {\color{b}adjacent} transitions in $\sigma$ is $\mu$.

{\color{b}
To verify \eqref{eq:nflow}, we observe that the total count $\sum_{i=0}^{\mu}\mathbf{1}\{x_i^{\sigma} = z\}$
can be computed two ways to give
\begin{align*}
\mathbf{1}\{x_0^{\sigma} = z\}+\sum_{y\in \{0,1\}^{k}} N_{y,z}^{\sigma}
\;= \; \mathbf{1}\{x_\mu^{\sigma} = z\}+\sum_{y'\in \{0,1\}^{k}} N_{z,y'}^{\sigma}
\end{align*}
Subtracting $N^\sigma_{z,z}$ from both sides yields \eqref{eq:nflow}.}\end{proof}

There are actually only $2^k$ linearly independent equations among the $2^k+1$ equations in \eqref{eq:nflow}--\eqref{eq:nsum}, as can be seen from the proof of Lemma \ref{lem:redundant} below. 
To ensure a non-degenerate limit when applying the central limit theorem, we utilize these $2^k$ linearly independent equations to remove $2^k$ redundant variables. Specifically, we remove the transition counts corresponding to the pairs
$\{(\vec{1},\,z):\,z\in \{0,1\}^{k}\} 
$, where $\vec{1}=(1,\cdots,1)\in \{0,1\}^{k}$ is the all-$1$ string.
\begin{lemma}[Redundancy]
\label{lem:redundant}
For any $\sigma \in \{0,1\}^{(\mu+1)k}$ and $\mu\in\mathbb{N}$, the vector $\big(x_0^\sigma, x_\mu^\sigma,  (N_{y,z}^\sigma)_{y,z \in \{0,1\}^{k}}\big)$ is a function of 
$\big(x_0^\sigma, x_\mu^\sigma, (N_{y,z}^\sigma)_{(y,z) \in \mathcal{H}}\big)$, where 
\begin{align*}
    \mathcal{H}
    = \left\{
    (y,z) \in \{0,1\}^{k} \times \{0,1\}^{k} : y \neq \vec{1}
    \right\}.
\end{align*}
\end{lemma}
\begin{proof}
It suffices to show that
for any $(y,z) \notin \mathcal{H}$, we can write $N_{y,z}^{\sigma}$ as a function of $x_0^\sigma, x_\mu^\sigma, \mu, $ and $(N_{y,z}^\sigma)_{(y,z) \in \mathcal{H}}$.  We do this first for $N_{\vec{1},z}^{\sigma}$ where $z \ne \vec{1}$, and then for $N_{\vec{1},\vec{1}}^{\sigma}$.

Among the $2^k$ equations in \eqref{eq:nflow}, each one indexed by $z \ne \vec{1}$ has exactly one variable in $\mathcal{H}^{c}$, namely $N_{\vec{1},z}^{\sigma}$. Precisely, \eqref{eq:nflow} gives 
\begin{align*}
    N_{\vec{1},z}^{\sigma} = \ind\{x_{\mu}^{\sigma} = z\} - \ind\{x_0^{\sigma} = z\} 
    + \sum_{y' \ne z}N_{z,y'}^{\sigma} - \sum_{y \ne z, \vec{1}} N_{y,z}^{\sigma},
\end{align*} 
in which all terms on the right come from $\mathcal{H}$. Hence  $N_{\vec{1},z}^{\sigma}$  can be written as a function of the required variables for each $z\neq \vec{1}$.

The variable $N_{\vec{1},\vec{1}}^{\sigma}$ is featured only in equation \eqref{eq:nsum}, and we obtain \begin{align*}
    N_{\vec{1},\vec{1}}^{\sigma} = \mu - \sum_{(y,z) \ne (\vec{1},\vec{1})}N_{y,z}^{\sigma}.
\end{align*} 
\end{proof}

For points $V\in\{A, B',C'\}$ on the trees, we let
\begin{equation}\label{Def:NZ}
 Z'_V = \left((x_0^{\sigma_V},x_1^{\sigma_V}),\,(x_{\mu-1}^{\sigma_V}, x_\mu^{\sigma_V}),\, N_{\mathcal{H}}^{\sigma_V}  \right) \quad \text{where} \quad
 N_{\mathcal{H}}^{\sigma_V}=(N_{y,z}^{\sigma_V})_{(y,z) \in \mathcal{H}}.
\end{equation}
\begin{lemma}[Reduction to non-redundant transitions]
\label{lem:red-3}
Suppose $m=(\mu+1)k$ for some $\mu\in\mathbb{N}$. Then
\begin{align*}
\left\|\P^{(1)}_{Z_A, Z_{B'}, Z_{C'}} - \P^{(2)}_{Z_A, Z_{B'}, Z_{C'}}\right\|_\tv
\leq 
\left\|\P^{(1)}_{Z'_A, Z'_{B'}, Z'_{C'}} - \P^{(2)}_{Z'_A, Z'_{B'}, Z'_{C'}}\right\|_\tv.
\end{align*}
\end{lemma}

\begin{proof}
From Lemma \ref{lem:redundant}, $(Z_{A},Z_{B'},Z_{C'})$ is a function of 
$(Z_{A}',Z_{B'}',Z_{C'}')$.  Then the result follows from Lemma \ref{lem:tv-f}.
\end{proof}

\subsubsection{
{\color{b}Final reduction step}}

By Lemmas~\ref{lem:red-1},~\ref{L:reduc_hatf},~\ref{lem:red-2} and~\ref{lem:red-3} above, together with
the second equality of Lemma~\ref{lem:tv-pointwise} in the appendix, {\color{b}to establish Theorem \ref{T:main}} it suffices to prove that 
\begin{align}
\label{eq:pointwise-zprime}
{\color{b}\liminf_{\mu\to\infty}}
\sum_{z'_A, z'_{B'}, z'_{C'}}
\P^{(1)}_{Z'_A, Z'_{B'}, Z'_{C'}}(z'_A, z'_{B'}, z'_{C'})\land \P^{(2)}_{Z'_A, Z'_{B'}, Z'_{C'}}(z'_A, z'_{B'}, z'_{C'}) > 0,
\end{align}
where the sum is taken over the set $\left(\{0,1\}^{2k}\times\{0,1\}^{2k}\times  \{0,1,\cdots,\mu\}^{\mathcal{H}} \right)^3$, and $m=(\mu+1)k$. 

\medskip

Our final reduction step in this section is to condition on the event 
\begin{equation}\label{Def:EventE}
\widetilde{\mathcal{E}} 
= 
\left\{
{\color{b}
(x_0^{\sigma_A},x_1^{\sigma_A}) = (x_0^{\sigma_{B'}},x_1^{\sigma_{B'}}) = (x_0^{\sigma_{C'}},x_1^{\sigma_{C'}})} = (\vec{0},\vec{0})
 \right\},
\end{equation}
where {\color{b2}$x^{\sigma_V}_j\in \{0,1\}^{k}$} are the {\color{b}adjacent} $k$-mers in the sequence $\sigma_V$ at point $V\in \{A, B',C'\}$, defined in \eqref{Def:NonOverlap}.
Precisely, for $i\in\{1,2\}$ we 
let $\widetilde{\P}^{(i)}=\widetilde{\P}^{(i),m}$  be the  conditional measures under  $\P^{(i)}=\P^{(i),m}$ given the event $\widetilde{\mathcal{E}}$.
Then 
\begin{align*}
&\sum_{z'_A, z'_{B'}, z'_{C'}}
\P^{(1)}_{Z'_A, Z'_{B'}, Z'_{C'}}(z'_A, z'_{B'}, z'_{C'})\land \P^{(2)}_{Z'_A, Z'_{B'}, Z'_{C'}}(z'_A, z'_{B'}, z'_{C'})\\ 
& \geq  \sum_{z'_A, z'_{B'}, z'_{C'}} \left(
\widetilde{\P}^{(1)}_{Z'_A, Z'_{B'}, Z'_{C'}}(z'_A, z'_{B'}, z'_{C'})
\,\P^{(1)}[\widetilde{\mathcal{E}}]\right)
\land \left(
\widetilde{\P}^{(2)}_{Z'_A, Z'_{B'}, Z'_{C'}}(z'_A, z'_{B'}, z'_{C'})
\,\P^{(2)}[\widetilde{\mathcal{E}}] \right)\\ 
&\geq 
\,c_1
\sum_{z'_A, z'_{B'}, z'_{C'}}
\widetilde{\P}^{(1)}_{Z'_A, Z'_{B'}, Z'_{C'}}(z'_A, z'_{B'}, z'_{C'})
\land 
\widetilde{\P}^{(2)}_{Z'_A, Z'_{B'}, Z'_{C'}}(z'_A, z'_{B'}, z'_{C'}),
\end{align*}
where $c_1:=\P^{(1)}[\widetilde{\mathcal{E}}]
\land
\,\P^{(2)}[\widetilde{\mathcal{E}}]$ is positive and does not depend on $\mu$.


Hence, to show~\eqref{eq:pointwise-zprime} it suffices to prove that
\begin{align}
\label{eq:pointwise-ztilde}
{\color{b}\liminf_{\mu\to\infty}}
\sum_{(z'_A, z'_{B'}, z'_{C'})\in (\mathcal{S}^{\mu}_0)^3}
\widetilde{\P}^{(1)}_{Z'_A, Z'_{B'}, Z'_{C'}}(z'_A, z'_{B'}, z'_{C'})\land \widetilde{\P}^{(2)}_{Z'_A, Z'_{B'}, Z'_{C'}}(z'_A, z'_{B'}, z'_{C'}) > 0,
\end{align}
where  $\mathcal{S}^{\mu}_0:=
\{(\vec{0},\vec{0})\}\times\{0,1\}^{2k} \times \{0,1,\cdots,\mu\}^{\mathcal{H}}$.
\medskip

For the rest of the proof, we shall 
establish \eqref{eq:pointwise-ztilde} by obtaining suitable lower bounds on the probabilities in \eqref{eq:pointwise-ztilde} through a local limit theorem.

\subsection{Applying a local limit theorem}\label{S:LocalCLT}

It will be convenient to consider infinite
sequences, since we shall employ a local limit theorem as the sequence length tends to infinity (i.e., $\mu\to\infty$).
Let $\P^{(i),\infty}$ be
the probability measure of the CFN model on $T_i$  
with infinite sequence length
and $\widetilde{\P}^{(i),\infty}$
be the conditional measure under  $\P^{(i),\,\infty}$, given the event $\widetilde{\mathcal{E}}$.

\subsubsection{Pairs of triplets as a Markov chain}\label{sec:asMC}

We shall apply Doeblin's method (see e.g.~\cite{culanovski1961twenty}).  
For  $V\in\{A, B', C'\}$, we let $\sigma_V(n) = (x_0^V,\ldots,x_{n}^V)\in \{0,1\}^{k(n+1)}$ be the first $n+1$ {\color{b}adjacent} $k$-mers of $\sigma_V$, 
where  $0\leq n\leq \mu$ if $\sigma_V$ has length $(\mu+1)k$ and $n\in\ZZ_+$ if $\sigma_V\in\{0,1\}^{\mathbb{N}}$ has infinite length.
For all such $n$,  we consider the triples
\begin{equation}\label{Def:triples}
\vec{X}_n=(x^A_n,\,x^{B'}_n,\,x^{C'}_n)\in\{0,1\}^{3k}.
\end{equation}

Under $\widetilde{\P}^{(i),\infty}$, $\{\vec{X}_n\}_{n\in\ZZ_+}$ is a sequence of independent random vectors and the pairs $\vec{M}_n=(\vec{X}_n,\,\vec{X}_{n+1})$ form a Markov chain with a finite state space. This Markov chain is irreducible since the support of $(\vec{X}_n,\,\vec{X}_{n+1})$ is all of $\{0,1\}^{3k}\times\{0,1\}^{3k}$ for all $n$.
{\color{b}The stationary distribution $\Theta_{\vec{M}}$ of $\{\vec{M}_n\}_{n\in\ZZ_+}$ is 
\begin{equation}\label{Theta1}
\Theta_{\vec{M}}(\vec{y},\vec{z}) = \widetilde{\P}^{(i),\infty}(\vec{X}_2=\vec{y})\,\widetilde{\P}^{(i),\infty}(\vec{X}_2=\vec{z}),\quad\text{for }\vec{y},\,\vec{z}\in \{0,1\}^{3k}.
\end{equation}}
Let $\tau_0 = 0$, let $\tau_1$ be {\color{b}the first $n>0$} such that $\vec{M}_n=(\vec{0},\vec{0})$ and in general, for $\ell \geq 1$, let 
\begin{equation}\label{Def:tau}
    \tau_{\ell} = \inf\{n > \tau_{\ell-1}: \vec{M}_n = (\vec{0},\vec{0})\},
\end{equation}
where an infimum over an empty set is $+\infty$ by convention.



The connection between $\widetilde{\P}^{(i)}=\widetilde{\P}^{(i),m}$ and $\widetilde{\P}^{(i),\infty}$ that we will need is given by Lemma \ref{L:minfinity} below. We let
\begin{equation*}
N_{y,z}^{V}(n)=N_{y,z}^{\sigma_V(n)}=\sum_{j=0}^{n-1}\ind\{x_j^V = y,\,x_{j+1}^V=z\}
\end{equation*}
be the number of {\color{b}adjacent} transitions from $y$ to $z$ up to $x^V_n$, as in \eqref{Def:N_yz}, with the convention that $N_{y,z}^{V}(0)=0$.
\begin{lemma}[Infinite sequences]
\label{L:minfinity}
For all $\mu\in\ZZ_+$, $k\in\mathbb{N}$, $\ell\in\{1,2,\cdots, \mu\}$, $(a,b,c)\in \ZZ_+^3$ and $i\in\{1,2\}$, the event
\[
\left\{\tau_{\ell}=\mu \;,\;
\left(N_{\mathcal{H}}^{A}(\tau_{\ell})
,\, N_{\mathcal{H}}^{B'}(\tau_{\ell})
,\, N_{\mathcal{H}}^{C'}(\tau_{\ell})
\right)
=(a,b,c)\right\}
\]
has the same probability under $\widetilde{\P}^{(i),m}$ and $\widetilde{\P}^{(i),\infty}$, where $m=(\mu+1)k$.
\end{lemma}

This lemma follows directly from the construction of the CFN model, in which non-overlapping, adjacent $k$-mers are independent. 
The rest of Section \ref{S:LocalCLT} concerns infinite sequences.

\subsubsection{Independent excursions and a multivariate local CLT}

We extract i.i.d.~random variables from excursions of the Markov chain $\vec{M}$.
Define, for  $V\in\{A, B', C'\}$,
\begin{align*}
Y_V(\ell)=
N_{\mathcal{H}}^{V}(\tau_{\ell})-N_{\mathcal{H}}^{V}(\tau_{\ell-1})
= \big(N_{y,z}^{V}(\tau_{\ell})
- N_{y,z}^{V}(\tau_{\ell-1})\big)_{(y,z) \in \mathcal{H}} 
\end{align*}
and let
\begin{align*}
\mathbf{Y}(\ell)
= \big(\tau_{\ell} - \tau_{\ell-1},Y_A(\ell), Y_{B'}(\ell),Y_{C'}(\ell)\big).
\end{align*}

Note that these random vectors take values in $\mathbb{N} \times \left( \ZZ_+^{\mathcal{H}}\right)^3 \subset \ZZ_+^{d}$ where $d=1+3(2^{2k}-2^k)$, because $|\mathcal{H}|=2^{2k}-2^k$. 
\begin{lemma}
\label{L:minfinity2}
The vectors $\{\mathbf{Y}(\ell)\}_{\ell=1}^{\infty}$ are
 i.i.d. under $\widetilde{\P}^{(i),\infty}$ for both $i=1$ and $2$.
Further, their partial sum is equal to
\begin{equation}\label{partialsumY}
\sum_{\iota =1}^\ell \mathbf{Y}(\iota)= \left(\tau_{\ell},\, N_{\mathcal{H}}^{A}(\tau_{\ell})
,\, N_{\mathcal{H}}^{B'}(\tau_{\ell})
,\, N_{\mathcal{H}}^{C'}(\tau_{\ell})
\right).
\end{equation}
\end{lemma}

\begin{proof}
The first statement is obvious from the construction of the CFN model. The equality \eqref{partialsumY} follows from the definitions of $\mathbf{Y}$ and the conventions $\tau_0=N_{y,z}^{V}(0)=0$.
\end{proof}



\medskip

We will apply a multivariate local CLT of Davis and McDonald~\cite[Theorem 2.1]{Davis1995}
to  the i.i.d. vectors 
$\{\mathbf{Y}(\ell)\}_{\ell=1}^{\infty} \subset \ZZ_+^d$ 
under $\widetilde{\P}^{(i),\infty}$. 
Theorem 2.1 of \cite{Davis1995} works for an array of independent vectors. Here we need only a sequence of i.i.d.~vectors so we state this result for the case of i.i.d. vectors in $\ZZ^d$. 
\begin{theorem} {\rm \cite[Theorem 2.1]{Davis1995}} \label{T:DMCLT}
Let $\{\mathbf{X}_j\}_{j=1}^{\infty}$ be a sequence of independent  $\mathbb{Z}^{d}$-valued random variables with a common probability mass function $f$,  finite mean $\mathbf{m}\in \RR^d$ and {\color{b}covariance matrix $\bSigma\in \RR^{d\times d}$. } 
Suppose the following hold:
\begin{enumerate}
    \item[(a)] {\color{b}For all $r\in\{1,2,\ldots,d\}$, there exists $\mathbf{x}_r \in \mathbb{Z}^{d}$ such that
    $$f(\mathbf{x}_r) \land f(\mathbf{x}_r + \mathbf{e}_r)\,>\,0,$$ 
    where $\mathbf{e}_r\in \mathbb{Z}^{d}$ is the $r$-th standard basis vector.}
     {\color{b}\item[(b)] 
   The sequence}
    $\frac{\textbf{S}_{\ell}-\ell\,\mathbf{m}}{\sqrt{\ell}}$  converges in distribution to the multivariate normal distribution $\mathcal{N}({\bf 0},\,\bSigma)$ as $\ell \to\infty$, {\color{b}where $\textbf{S}_\ell=\sum_{j=1}^\ell\textbf{X}_j$.}
\end{enumerate}
Then the following uniform convergence holds as $\ell \rightarrow \infty$: 
\begin{align*}
    \sup_{\mathbf{y} \in \mathbb{Z}^{d}}\left| \ell^{d/2}\,\P[\textbf{S}_\ell = \mathbf{y}] - \varphi\left(\frac{\mathbf{y}-\ell\,\mathbf{m}}{\sqrt{\ell}}\right)\right| \rightarrow 0,
\end{align*}
where  $\varphi$
is the probability density function of the multivariate normal distribution $\mathcal{N}({\bf 0},\,\bSigma)$.
\end{theorem}

\medskip

{\color{b}  Condition (a) of Theorem \ref{T:DMCLT} implies that  the  multivariate normal distribution $\mathcal{N}({\bf 0},\,\bSigma)$ is non-degenerate.}
\begin{lemma}
\label{lem:pd}
{\color{b}
Let $f$ be a 
probability mass function on $\ZZ^d$ with  finite mean and covariance matrix $\bSigma\in \RR^{d\times d}$.
Assume condition (a) of Theorem \ref{T:DMCLT} holds. Then $\bSigma$ is positive definite.} 
\end{lemma}
\begin{proof}
Let $X$ and $Y$ be two independent random vectors with distribution $f$. Then the covariance matrix of $X$ can be written as 
$$
\E[(X - \E[X]) (X - \E[X])^T]
= (1/2)\E[(X-Y) (X-Y)^T].
$$ 
{\color{b}Let $\mathbf{x}_r$ be as in condition (a) of Theorem \ref{T:DMCLT}.}
Then for any nonzero vector $\mathbf{z} = (z_1,\ldots,z_d) \neq 0$ with, say, $z_r \neq 0$, we have
\begin{align*}
\mathbf{z}^T \,\E[(X-Y) (X-Y)^T] \,\mathbf{z}
&= \E\big[\big(\mathbf{z}^T (X-Y) \big)^2\big]\\
&\geq f(\mathbf{x}_r) f(\mathbf{x}_r + \mathbf{e}_r) z_r^2\\ 
&> 0,
\end{align*}
where the expression on the second line is the contribution to the expectation from the event that $X = \mathbf{x}_r + \mathbf{e}_r$ and
$Y = \mathbf{x}_r$, and the third line follows {\color{b}from condition (a) of Theorem \ref{T:DMCLT}.} Note that we used that
each term contributing to the expectation is non-negative.
\end{proof}

\medskip

We {\color{b}shall} apply Theorem \ref{T:DMCLT} to  the 
i.i.d.~vectors 
$\{\mathbf{Y}(\ell)\}_{\ell=1}^{\infty} \subset \ZZ_+^d$ 
under $\widetilde{\P}^{(i),\infty}$, for each of $i\in\{1,2\}$.

\subsubsection{Checking conditions of the local CLT}\label{S:Checkabc}

In this section we verify {\color{b}
that the i.i.d.~vectors 
$\{\mathbf{Y}(\ell)\}_{\ell=1}^{\infty} \subset \ZZ_+^d$ satisfy all conditions of Theorem \ref{T:DMCLT}. We also show that they have the same mean under $\widetilde{\P}^{(1),\infty}$ and $\widetilde{\P}^{(2),\infty}$.
}
For this we let
$f^{(i)}$ be the probability mass function of 
$$\mathbf{Y}(1)
\,=\, \left(\tau, Y_A(1), Y_{B'}(1), Y_{C'}(1)\right)\,=\,
\left(\tau,\,N_{\mathcal{H}}^{A}(\tau),\,N_{\mathcal{H}}^{B'}(\tau),\,N_{\mathcal{H}}^{C'}(\tau)\right)$$ 
under $\widetilde{\P}^{(i),\infty}$ for $i\in\{1,2\}$, where  $\tau=\tau_1$ is defined in  \eqref{Def:tau}.


\begin{lemma}[Non-degeneracy]
\label{lem:davis}
{\color{b}
The distributions $f^{(1)}$ and $f^{(2)}$ both satisfy condition (a) of Theorem \ref{T:DMCLT}.}
\end{lemma}
\begin{proof}
Fix $i\in\{1,2\}$.
The proof relies crucially on the construction of the set $\mathcal{H}$ in Lemma \ref{lem:redundant}. 
We write a point in $\ZZ_+^d$ as
\[
\mathbf{x}=\left(t,\;(n^{A}_{yz},\,n^{B'}_{yz},\,n^{C'}_{yz})_{yz\in\mathcal{H}}\right), \quad \text{where } t\in \ZZ_+ \text{ and } n^{A}_{yz},\,n^{B'}_{yz},\,n^{C'}_{yz}\in \ZZ_+.
\]

Recall that $\vec{0}$ and $\vec{1}$ refer to the all-0 and all-1 $k$-mers.
A {\color{b}sequence} of {\color{b}adjacent} $k$-mer triples starting and ending with 
$(\vec{0},\vec{0},\vec{0})\, (\vec{0},\vec{0},\vec{0})$ will give rise to a unique point in $\ZZ_+^d$, in which $t$ is the length of the {\color{b}sequence} and $n^{V}_{yz}$ counts the number of $yz$-transitions. 
By the definition of $\mathcal{H}$, we are not counting the transition from $\vec{1}$ to $z$ for any $z\in \{0,1\}^k$.

\medskip
For $r=1$ (corresponding to the $t$-coordinate), 
we consider the $k$-mer triple cycles of
\begin{align*}
\mathcal{C}
= &
(\vec{0},\vec{0},\vec{0}),(\vec{0},\vec{0},\vec{0}),
(\vec{1},\vec{1},\vec{1}),
(\vec{0},\vec{0},\vec{0}),(\vec{0},\vec{0},\vec{0}) \qquad \text{and}\\
\mathcal{C}^+
= &
(\vec{0},\vec{0},\vec{0}),
(\vec{0},\vec{0},\vec{0}), 
(\vec{1},\vec{1},\vec{1}),
(\vec{1},\vec{1},\vec{1}),
(\vec{0},\vec{0},\vec{0}),
(\vec{0},\vec{0},\vec{0}).
\end{align*}
They give rise to $\mathbf{x}_r$ and $\mathbf{x}_r + \mathbf{e}_r$ respectively, where
we take
 $\mathbf{x}_r$ to be the point {\color{b}in $\ZZ_+^d$} such that $t=3$ and 
\begin{equation}\label{x_r}
(n^{A}_{yz},\,n^{B'}_{yz},\,n^{C'}_{yz})=
\begin{cases}
(2,2,2) \quad \text{if } (y,z)=(\vec{0},\vec{0}) \\
(1,1,1) \quad \text{if } (y,z)=(\vec{0},\vec{1}) \\
(0,0,0) \quad \text{if } (y,z)\in \mathcal{H}\setminus \{(\vec{0},\vec{0}),\,(\vec{0},\vec{1})\},
\end{cases}
\end{equation}
and $\mathbf{x}_r + \mathbf{e}_r=\left(4,\;(n^{A}_{yz},\,n^{B'}_{yz},\,n^{C'}_{yz})_{yz\in\mathcal{H}}\right)$. Recall that
\begin{align*}
    \mathcal{H}
    = \left\{
    (y,z) \in \{0,1\}^{k} \times \{0,1\}^{k} : y \neq \vec{1}
    \right\},
\end{align*}
so that, in particular, the transitions 
$(\vec{1},\vec{1})$ are not counted.
Then
\begin{align*}
f^{(i)}(\mathbf{x}_r) \geq  &\;
\widetilde{\P}^{(i),\infty}\left(
(\vec{X}_n)_{n=0}^3=\mathcal{C}
\right) >0\qquad \text{ and}\\
f^{(i)}(\mathbf{x}_r+ \mathbf{e}_r) \geq  &\;
\widetilde{\P}^{(i),\infty}\left(
(\vec{X}_n)_{n=0}^4=\mathcal{C}^+
\right) >0,
\end{align*}
where $\vec{X}_n=(x^A_n,\,x^{B'}_n,\,x^{C'}_n)\in\{0,1\}^{3k}$ as defined in \eqref{Def:triples}.

\medskip
For $r>1$, we first suppose $r$ corresponds to the coordinate $n^A_{ab}$ where $(a,b)\in \mathcal{H}$. 
The cycles
\begin{align*}
\mathcal{C}^A_{ab}
= &
(\vec{0},\vec{0},\vec{0}),
(\vec{0},\vec{0},\vec{0}), 
(\vec{1},\vec{1},\vec{1}),
(\vec{1},\vec{1},\vec{1}),
(b,\vec{1},\vec{1}),
(\vec{0},\vec{0},\vec{0}),
(\vec{0},\vec{0},\vec{0})
 \qquad \text{and}\\
\mathcal{C}^{A+}_{ab}
= &
(\vec{0},\vec{0},\vec{0}),
(\vec{0},\vec{0},\vec{0}), 
(\vec{1},\vec{1},\vec{1}),
(a,\vec{1},\vec{1}),
(b,\vec{1},\vec{1}),
(\vec{0},\vec{0},\vec{0}),
(\vec{0},\vec{0},\vec{0})
\end{align*}
give rise to 
$\mathbf{x}_r$ and $\mathbf{x}_r + \mathbf{e}_r$ respectively, where $\mathbf{x}_r$ is the point on $\ZZ_+^d$ such that $t=5$ and \eqref{x_r} holds. Hence both $f^{(i)}(\mathbf{x}_r)$ and  $f^{(i)}(\mathbf{x}_r+ \mathbf{e}_r)$ are positive, as before.

The proof for coordinates $n^{B'}_{ab}$ is the same, except that we replace $(a,\vec{1},\vec{1})$ by  $(\vec{1},a,\vec{1})$ and  $(b,\vec{1},\vec{1})$ by  $(\vec{1},b,\vec{1})$. The proof for coordinates $n^{C'}_{ab}$ follows similarly. The proof is complete.
\end{proof}

To verify condition (b) of Theorem \ref{T:DMCLT},
{\color{b}
we let $\mathbf{m}^{(i)}$ and $\bSigma^{(i)}$ be respectively the mean and the covariance matrix of $\mathbf{Y}(1)$ under $\widetilde{\P}^{(i),\infty}$. 
We also let $\textbf{S}_{\ell}=\sum_{j=1}^{\ell}\mathbf{Y}(j)$.}
\begin{lemma}
\label{lem:var}
{\color{b}
For $i\in\{1,2\}$, under $\widetilde{\P}^{(i),\infty}$,
the sequence
$\frac{\textbf{S}_{\ell}-\ell\,\mathbf{m}^{(i)}}{\sqrt{\ell}}$  converges in distribution to the multivariate normal distribution $\mathcal{N}({\bf 0},\,\bSigma^{(i)})$ as $\ell \to\infty$.}
\end{lemma}
\begin{proof}
Fix $i\in\{1,2\}$.
Observe that
$\mathbf{Y}(1) \leq 
(\tau_1,\tau_1,\ldots,\tau_1)$ coordinate-wise.
Moreover, by construction, $\tau_1$ is geometric and therefore has finite first and second moments. {\color{b}Hence $\mathbf{m}^{(i)}$ is finite and  $\widetilde{\E}^{(i),\infty}[\|\mathbf{Y}(\ell)\|^2]<\infty$, from which we have that the entries of $\bSigma^{(i)}$ are finite and hence $|\det(\bSigma^{(i)})|<\infty$.} Also $\bSigma^{(i)}$ is positive definite by  Lemmas~\ref{lem:pd} and~\ref{lem:davis}. 
{\color{b}The claim} follows from the multivariate central limit theorem (see, e.g.,~\cite[Section 3.10]{durrett2019probability}).
\end{proof}

Due to symmetry between $T_1$ and $T_2$, as well as the independence of {\color{b}non-overlaping, adjacent} $k$-mers under the CFN model,
the expectations  are the same, as we show formally next. 
\begin{lemma}[Expectation]
\label{lem:expec}
{\color{b}The equality $\mathbf{m}^{(1)}=\mathbf{m}^{(2)}\in \RR_+^d$ holds.}
\end{lemma}

\begin{proof}
By symmetry \eqref{topo_ABC}, we have $\widetilde{\P}^{(1),\infty}_{(\tau,\sigma_A,\sigma_{B'},\sigma_{C'})}=\widetilde{\P}^{(2),\infty}_{(\tau,\sigma_A,\sigma_{C'},\sigma_{B'})}$. 
Hence 
$$\widetilde{\E}^{(1),\infty}[(\tau, N_{\mathcal{H}}^{A}(\tau))]=\widetilde{\E}^{(2),\infty}[(\tau, N_{\mathcal{H}}^{A}(\tau))],$$ and
$$\widetilde{\E}^{(1),\infty}[( N_{\mathcal{H}}^{B'}(\tau), N_{\mathcal{H}}^{C'}(\tau))]=\widetilde{\E}^{(2),\infty}[( N_{\mathcal{H}}^{C'}(\tau), N_{\mathcal{H}}^{B'}(\tau))].$$
It remains to show that
\begin{equation}\label{Symmetry}
  \widetilde{\E}^{(i),\infty}[ N_{\mathcal{H}}^{B'}(\tau)]=\widetilde{\E}^{(i),\infty}[N_{\mathcal{H}}^{C'}(\tau)] \quad \text{for } i\in\{1,2\}.
\end{equation}
{\color{b}While $\widetilde{\P}^{(i),\infty}_{\sigma_{B'}}=\widetilde{\P}^{(i),\infty}_{\sigma_{C'}}$, Eq.~\eqref{Symmetry} is not immediately clear because $\tau$ depends on all three sequences.} 

{\color{b}
Using the notation of Section~\ref{sec:asMC}, for arbitrary $(y,z)\in \{0,1\}^{k}\times \{0,1\}^{k}$, we have
\begin{equation}\label{NrepresentationM}
N_{(y,z)}^{B'}(\tau)  = \sum_{(\vec{y},\vec{z}):\,(y_2,z_2)=(y,z)}
 \sum_{j=0}^{\tau-1} 1_{\vec{M}_j=(\vec{y},\,\vec{z})},
\end{equation}
where the sum is 
over the set of $(\vec{y},\vec{z})$
with $y_2=y$ and $z_2=z$, with $\vec{y}=(y_1,y_2,y_3)\in \{0,1\}^{3k}$ and 
$\vec{z}=(z_1,z_2,z_3)\in \{0,1\}^{3k}$. 
Using standard Markov chain results (e.g.,~\cite[Chapter 5]{durrett2019probability}), 
\begin{equation}\label{Theta2}
\widetilde{\E}^{(i),\infty}\left[ \sum_{j=0}^{\tau-1} 1_{\vec{M}_j=(\vec{y},\,\vec{z})}\right] 
=\widetilde{c}\,\Theta_{\vec{M}}(\vec{y},\vec{z}),
\end{equation}
where $\widetilde{c}=\widetilde{\E}^{(i),\infty}[\tau]\in(0,\infty)$ and the stationary distribution $\Theta_{\vec{M}}(\vec{y},\vec{z})$ was computed in~\eqref{Theta1}.
Combining \eqref{Theta1}, \eqref{NrepresentationM}, and \eqref{Theta2}, we have
\begin{align*}
\widetilde{\E}^{(i),\infty}\left[
N_{y,z}^{B'}(\tau) \right] &=    \widetilde{c}\,\sum_{(\vec{y},\vec{z}):\,(y_2,z_2)=(y,z)}
\widetilde{\P}^{(i),\infty}(\vec{X}_2=\vec{y})\,\widetilde{\P}^{(i),\infty}(\vec{X}_2=\vec{z})\\
&= \widetilde{c}\;\widetilde{\P}^{(i),\infty}(x^{B'}_2=y)\,\widetilde{\P}^{(i),\infty}(x^{B'}_2=z) 
\end{align*}
and, similarly for $C'$,
\begin{align*}
\widetilde{\E}^{(i),\infty}\left[
N_{y,z}^{C'}(\tau) \right] &=    \widetilde{c}\;\widetilde{\P}^{(i),\infty}(x^{C'}_2=y)\,\widetilde{\P}^{(i),\infty}(x^{C'}_2=z).
\end{align*}
The two displayed equations are the same since  $\widetilde{\P}^{(i),\infty}_{\sigma_{B'}}=\widetilde{\P}^{(i),\infty}_{\sigma_{C'}}$.}
\end{proof}

\subsubsection{Applying the local CLT}

By Lemmas~\ref{lem:davis} and~\ref{lem:var}, we can apply  Theorem \ref{T:DMCLT} to  the 
i.i.d.~vectors 
$\{\mathbf{Y}(j)\}_{j=1}^{\infty}$ to obtain the following lower bound. Recall that $\mathbf{m}^{(1)}=\mathbf{m}^{(2)}$ by  Lemma \ref{lem:expec}, and let $\mathbf{m}=\mathbf{m}^{(i)}$.
{\color{b}Recall also that $\textbf{S}_{\ell}=\sum_{j=1}^{\ell}\mathbf{Y}(j)$.}


\begin{lemma}[Uniform lower bound]
\label{lem:local-clt}
There exist constants $c_1,\,c_2\in(0,\infty)$ such that
\begin{equation*}
\inf_{y\in\mathcal{Y}^{(i)}_{\ell}}\widetilde{\P}^{(i),\infty}[\textbf{S}_{\ell} = \mathbf{y}] \geq \frac{c_2}{\ell^{d/2}}
\end{equation*}
for all $\ell\geq c_1$ and $i\in\{1,2\}$, where 
\begin{equation}\label{Def:setYi}
    \mathcal{Y}^{(i)}_{\ell}:=\left\{\mathbf{y}\in \ZZ_+^d:\; \left(\mathbf{y}-\ell\,\mathbf{m}\right)^\mathrm{T}\left(\bSigma^{(i)} \right)^{-1}\,\left(\mathbf{y}-\ell\,\mathbf{m}\right) \leq 2\ell  \right\}.
\end{equation}
\end{lemma}
\begin{proof}

By Theorem \ref{T:DMCLT}, for $i\in\{1,2\}$, as $\ell\to\infty$, 
\begin{align}
    \sup_{\mathbf{y} \in \mathbb{Z}^{d}}\left| \ell^{d/2}\,\widetilde{\P}^{(i),\infty}[\textbf{S}_{\ell} = \mathbf{y}] - \varphi^{(i)}\left(\frac{\mathbf{y}-\ell\,\textbf{m}}{\sqrt{\ell}}\right)\right| \rightarrow 0. 
\end{align}
where  
\[
\varphi^{(i)}(\textbf{x})= \frac{\exp\left\{-\frac 1 2 {\mathbf x}^\mathrm{T}\left(\bSigma^{(i)} \right)^{-1}\,{\mathbf x}\right\}}{\sqrt{(2\pi)^d \,\det(\bSigma^{(i)})}}.
\]

Therefore, for arbitrary $\epsilon > 0$, there exists $\ell_{\epsilon}$ sufficiently large such that  for all integers $\ell\geq \ell_{\epsilon}$ and all $\mathbf{y}\in \ZZ^d$,
\begin{align*}
\widetilde{\P}^{(i),\infty}[\textbf{S}_{\ell} = \mathbf{y}] 
\geq&\, \frac{1}{\ell ^{d/2}}\left(\varphi^{(i)}\left(\frac{\mathbf{y}-\ell\,\textbf{m}}{\sqrt{\ell}}\right) - \epsilon\right)\\
= &\,\frac{1}{\ell ^{d/2}}\left(\frac{\exp\left\{-\frac 1 2 \left(\frac{\mathbf{y}-\ell\,\textbf{m}}{\sqrt{\ell}}\right)^\mathrm{T}\left(\bSigma^{(i)} \right)^{-1}\,\left(\frac{\mathbf{y}-\ell\,\textbf{m}}{\sqrt{\ell}}\right)\right\}}{\sqrt{(2\pi)^d \,\det(\bSigma^{(i)})}}- \epsilon\right).
\end{align*}
The bound in the definition of 
$\mathcal{Y}^{(i)}_{\ell}$  
gives
\begin{equation*}
\inf_{y\in\mathcal{Y}^{(i)}_{\ell}}\widetilde{\P}^{(i),\infty}[\textbf{S}_{\ell} = \mathbf{y}] \geq 
\frac{1}{\ell ^{d/2}}\left(\frac{e^{-1}}{\sqrt{(2\pi)^d \,[\det(\bSigma^{(1)})\vee \det(\bSigma^{(2)})]}}- \epsilon\right)
\end{equation*}
for all $\ell\geq \ell_{\epsilon}$.
The lemma follows by taking $\epsilon$ to be any fixed  number small enough that depends only on $\det(\bSigma^{(1)})\vee \det(\bSigma^{(2)})$.
\end{proof}

Observe that the bound in Lemma~\ref{lem:local-clt} is uniform over the set $ \mathcal{Y}^{(i)}_{\ell}$.
Our use of Lemma~\ref{lem:local-clt} below
will require a lower bound on the size of $ \mathcal{Y}^{(1)}_{\ell}\cap \mathcal{Y}^{(2)}_{\ell}$.
\begin{lemma}\label{L:sizeY}
Let  $\lambda^{(i)}_{min}$ be  the  minimal eigenvalues of  $\bSigma^{(i)}$. Then
\begin{equation}\label{subsetY}
 \left\{\mathbf{y}\in \ZZ_+^d:\,\|\mathbf{y}-\ell\mathbf{m}\|^2\leq 2\ell\,(\lambda^{(1)}_{min} \wedge \lambda^{(2)}_{min}) \right\}\subset   \mathcal{Y}^{(1)}_{\ell}\cap \mathcal{Y}^{(2)}_{\ell},
\end{equation}
where $\{\mathcal{Y}^{(i)}_{\ell}\}_{i=1}^2$ are defined in \eqref{Def:setYi}.
\end{lemma}

\begin{proof}
Note that $0<\lambda^{(i)}_{min}<\infty$ by Lemma \ref{lem:var}. Since $\lambda$ is an eigenvalue of $\bSigma^{(i)}$ if and only if $1/\lambda$ is  an eigenvalue of $\big(\bSigma^{(i)}\big)^{-1}$,
we have
\[
\left(\mathbf{y}-\ell\mathbf{m}\right)^\mathrm{T}\left(\bSigma^{(i)} \right)^{-1}\,\left(\mathbf{y}-\ell\mathbf{m}\right)\leq \frac{1}{\lambda^{(i)}_{min}}\,\|\mathbf{y}-\ell\mathbf{m}\|^2.
\]
This inequality implies \eqref{subsetY}.
\end{proof}

In fact, we will need to control the size of subsets of $\mathcal{Y}^{(1)}_{\ell}\cap \mathcal{Y}^{(2)}_{\ell}$ whose first coordinates are sufficiently close to 
their expectation.
Letting $m_1$ be the first coordinate of $\mathbf{m}$, by  Lemma \ref{lem:expec},  
\begin{equation}\label{Def:m1}
m_1=\widetilde{\E}^{(1),\infty}[\tau_1]=\widetilde{\E}^{(2),\infty}[\tau_1].
\end{equation}
We consider the following set of pairs $(\mu,\ell)$
\begin{equation}\label{Def:Lell}
\mathcal{L}=\left\{(\mu,\ell)\in \mathbb{N}^2:\, |\mu-\ell m_1|\leq c_3 \sqrt{\ell} \right\} \quad \text{where}\quad c_3= \sqrt{\lambda^{(1)}_{min} \wedge \lambda^{(2)}_{min}}.
\end{equation}
The next two lemmas concern bounds on 
the level sets
\[
\mathcal{L}|_{\ell}:=\{\mu\in \mathbb{N}:\,(\mu,\ell)\in \mathcal{L}\} 
\quad\text{and}\quad
\mathcal{L}|_{\mu}:=\{\ell\in \mathbb{N}:\,(\mu,\ell)\in \mathcal{L}\}.
\]

\begin{lemma}\label{L:sizeY2}
Let $\ZZ_{+}^{d}(\mu)$ be the subset of $\ZZ_{+}^{d}$ whose first coordinate is $\mu$. Then
\begin{equation}
    \inf_{\mu\in \mathcal{L}|_{\ell}}\big| \mathcal{Y}^{(1)}_{\ell}\cap \mathcal{Y}^{(2)}_{\ell} \cap \ZZ_{+}^{d}(\mu) \big|\geq c_4\,c_3^{d-1}\,\ell^{(d-1)/2}
\end{equation}
for all $\ell\in\mathbb{N}$, where $c_4\in(0,\infty)$ is a constant that depends only on $d$.
\end{lemma}

\begin{proof}
By Lemma~\ref{L:sizeY}, the set $\mathcal{Y}^{(1)}_{\ell}\cap \mathcal{Y}^{(2)}_{\ell}$ contains 
all integer points of 
$$
B_d(\ell\mathbf{m},\,c_3\sqrt{2\ell}) \cap \mathbb{R}_+^d,
$$
where $B_d(x,r):=\{y\in\mathbb{R}^d:\,\|y-x\|_{\mathbb{R}^d}\leq r\}$ is the 
$d$-dimensional Euclidean ball with center $x$ and radius $r$. By Lemma \ref{lem:expec},
$\mathbf{m}=(m_1,m_2,\cdots,m_d)\in \RR_+^d$.
Hence  $\mathcal{Y}^{(1)}_{\ell}\cap \mathcal{Y}^{(2)}_{\ell} \cap \ZZ_{+}^{d}(\mu)$ contains all integer points of
$\widetilde{B}(r_{\mu})\cap \mathbb{R}_+^d$, where 
\begin{align*}
\widetilde{B}(r_{\mu}):=  \{(\mu,\,y_2,\cdots,y_d)\in \mathbb{R}^d:\,\|(y_2,\cdots,y_d)-\ell(m_2,\cdots, m_d)\|_{\mathbb{R}^{d-1}}\leq r_{\mu}\} 
\end{align*}
with
\begin{equation}\label{rmu}
    r_{\mu}:=\sqrt{c_3^22\ell-(\mu-\ell m_1)^2}\geq c_3\sqrt{\ell}.
\end{equation}
The last inequality follows whenever $(\mu,\ell)\in \mathcal{L}$.

Since $\{m_2,\cdots,m_d\}$ are all non-negative, $\widetilde{B}(r_{\mu})\cap \mathbb{R}_+^d$ contains a $(d-1)$-dimensional cube with side length $\frac{r_{\mu}}{\sqrt{d-1}}\geq \frac{c_3\sqrt{\ell}}{\sqrt{d-1}}$ by \eqref{rmu}. This cube contains at least $c_4c_3^{d-1}\ell^{(d-1)/2}$ many integer points for some  $c_4\in (0,\infty)$ that depends only on $d$, uniformly for all $(\mu,\ell)\in \mathcal{L}$.
\end{proof}

Finally, the following lemma 
gives a lower bound on the cardinality of 
$\mathcal{L}|_{\mu}$.
\begin{lemma}
\label{lem:size-mu}
There exists a constant $c_5\in (0,\infty)$ that depend only on $m_1$ and $c_3$ such that
for $\mu$ large enough,
$$
\left[\frac{\mu}{m_1} - c_5 \sqrt{\mu},\; \frac{\mu}{m_1} + c_5 \sqrt{\mu} \right]
\subseteq \mathcal{L}|_{\mu},
$$
where $m_1=\widetilde{\E}^{(1),\infty}[\tau_1]=\widetilde{\E}^{(2),\infty}[\tau_1]$.
\end{lemma}
\begin{proof}
Suppose $\ell$ belongs to the interval on the left-hand side of the display in the statement of the lemma. Then $\ell\geq  \frac{\mu}{m_1} - c_5 \sqrt{\mu}$. Solving this quadratic inequality in $\sqrt{\mu}$ and then {\color{b}squaring gives}
\begin{align*}
\sqrt{\mu} \leq\, & \frac{c_5m_1+\sqrt{(c_5m_1)^2+4m_1\ell}}{2},
\end{align*}
and
\begin{align*}
\mu \leq\, & \ell m_1 + {\color{b}\frac{1}{2} (c_5m_1)^2} + \frac{c_5 m_1 \sqrt{(c_5m_1)^2+4m_1\ell}}{2}.
\end{align*}
From the last inequality, we see that $\mu\leq \ell m_1 +c_3 \sqrt{\ell}$ for all $\ell \geq 1$, provided that $c_5\in (0,\infty)$ is small enough (depending only on $c_3$ and $m_1$).

Similarly, by solving the inequality  $\ell\leq \frac{\mu}{m_1} + c_5 \sqrt{\mu}$ to yield
\begin{align*}
\sqrt{\mu} \geq\, & \frac{-c_5m_1+\sqrt{(c_5m_1)^2+4m_1\ell}}{2},
\end{align*}
and
\begin{align*}
\mu \geq\, & 
\ell m_1 + {\color{b}\frac{1}{2} (c_5m_1)^2} -\frac{ c_5m_1\sqrt{(c_5m_1)^2+4m_1\ell}}{2}.
\end{align*}
{\color{b}For sufficiently small $c_5\in (0,\infty)$ } (depending only on $c_3$ and $m_1$), we have
$\mu\geq \ell m_1 -c_3 \sqrt{\ell}$.

The desired subset relation is obtained.
\end{proof}

\subsection{Final bound on the total variation distance}\label{S:Final}

\begin{proof}[Proof of Theorem \ref{T:main}]

Now we finish the proof of Theorem \ref{T:main} by establishing \eqref{eq:pointwise-ztilde}. That is, we now show that
\begin{align}
\label{eq:pointwise-ztilde2}
{\color{b}\liminf_{\mu\to\infty}}\sum_{(z'_A, z'_{B'}, z'_{C'})\in (\mathcal{S}^{\mu}_0)^3}
\widetilde{\P}^{(1)}_{Z'_A, Z'_{B'}, Z'_{C'}}(z'_A, z'_{B'}, z'_{C'})\land \widetilde{\P}^{(2)}_{Z'_A, Z'_{B'}, Z'_{C'}}(z'_A, z'_{B'}, z'_{C'}) \;>\; 0,
\end{align}
where  $\mathcal{S}^{\mu}_0:=
\{(\vec{0},\vec{0})\}\times\{0,1\}^{2k} \times \{0,1,\cdots,\mu\}^{\mathcal{H}}$ and $m=(\mu+1)k$.
We further restrict the last pair of triples by considering $\mathcal{S}^{\mu}_{00}:=\{(\vec{0},\vec{0})\}\times \{(\vec{0},\vec{0})\} \times \{0,1,\cdots,\mu\}^{\mathcal{H}}$.
Since $\mathcal{S}^{\mu}_{00}\subset \mathcal{S}^{\mu}_0$, the sum $\sum_{(z'_A, z'_{B'}, z'_{C'})\in (\mathcal{S}^{\mu}_0)^3}$ on the left of 
\eqref{eq:pointwise-ztilde2} is bounded below by
\begin{align}
{\color{b}W_{00}=}\sum_{(z'_A, z'_{B'}, z'_{C'})\in (\mathcal{S}^{\mu}_{00})^3}
\widetilde{\P}^{(1)}_{Z'_A, Z'_{B'}, Z'_{C'}}(z'_A, z'_{B'}, z'_{C'})\land \widetilde{\P}^{(2)}_{Z'_A, Z'_{B'}, Z'_{C'}}(z'_A, z'_{B'}, z'_{C'}). \label{ZtoY1}
\end{align} 

As an element of $\mathcal{S}^{\mu}_{00}$,  $z'_V=\big((\vec{0},\vec{0}),\,(\vec{0},\vec{0}),\, N_V'\big)$ for some $N_V'\in\{0,1,\cdots,\mu\}^{\mathcal{H}}$, where $V\in\{A,B' ,C'\}$. Hence
\begin{align}
&\widetilde{\P}^{(i)}_{Z'_A, Z'_{B'}, Z'_{C'}}(z'_A, z'_{B'}, z'_{C'}) \notag\\
&=\sum_{\ell=1}^{\mu}
\widetilde{\P}^{(i)}\left\{
\left(N_{\mathcal{H}}^{\sigma_A}(\tau_{\ell})
,\, N_{\mathcal{H}}^{\sigma_{B'}}(\tau_{\ell})
,\, N_{\mathcal{H}}^{\sigma_{C'}}(\tau_{\ell})
\right)
=(N'_A, N'_{B'}, N'_{C'})\;,\; \tau_{\ell}=\mu \right\} \notag\\
&=\sum_{\ell=1}^{\mu}
\widetilde{\P}^{(i),\infty}\left\{
\left(N_{\mathcal{H}}^{\sigma_A}(\tau_{\ell})
,\, N_{\mathcal{H}}^{\sigma_{B'}}(\tau_{\ell})
,\, N_{\mathcal{H}}^{\sigma_{C'}}(\tau_{\ell})
\right)
=(N'_A, N'_{B'}, N'_{C'})\;,\; \tau_{\ell}=\mu \right\} \notag\\
&= \sum_{\ell=1}^{\mu}
\widetilde{\P}^{(i),\infty}\left\{
\sum_{j =1}^\ell \mathbf{Y}(j)=(\mu,\,N'_A, N'_{B'}, N'_{C'})
\right\}, \notag
\end{align}
where the second and the last equalities follow from Lemma \ref{L:minfinity} and \eqref{partialsumY}
respectively. Therefore, 
\begin{align}
{\color{b}W_{00}}\geq\,& \sum_{\ell =1}^{\mu} \sum_{\mathbf{y} \in \ZZ_{+}^{d}(\mu)}
\widetilde{\P}^{(1),\infty}\left\{
\sum_{j =1}^\ell \mathbf{Y}(j)=\mathbf{y}
\right\}\land \widetilde{\P}^{(2),\infty}\left\{
\sum_{j =1}^\ell \mathbf{Y}(j)=\mathbf{y}
\right\} \label{FinalBound1},
\end{align} 
{\color{b}where recall that $\ZZ_{+}^{d}(\mu)$ was defined in Lemma~\ref{L:sizeY2}.}

We further restrict the sums to be over 
$(\mu,\,N'_A, N'_{B'}, N'_{C'})\in \mathcal{Y}^{(1)}_{\ell}\cap\mathcal{Y}^{(2)}_{\ell}$ and $\ell \geq c_1$, where recall that $\mathcal{Y}^{(i)}_{\ell}$ and $c_1$ were defined in Lemma 
\ref{lem:local-clt}.
We obtain from Lemma 
\ref{lem:local-clt} that 
 the right-hand side of \eqref{FinalBound1}  is
\begin{align*}
&\geq \,\sum_{\ell\in [c_1,\,\mu]\cap \ZZ_+ }\; \sum_{\mathbf{y} \in \mathcal{Y}^{(1)}_{\ell}\cap\mathcal{Y}^{(2)}_{\ell} \cap \ZZ_{+}^{d}(\mu) }  \frac{c_2}{\ell^{d/2}}\\ 
&= c_2\sum_{\ell\in [c_1,\,\mu]\cap \ZZ_+ } \frac{\big|\mathcal{Y}^{(1)}_{\ell}\cap\mathcal{Y}^{(2)}_{\ell} \cap \ZZ_{+}^{d}(\mu)\big|}{\ell^{d/2}}\\
&\geq c_2\,c_3^{d-1}\,c_4\,\sum_{\ell\in [c_1,\,\mu]\cap \mathcal{L}|_{\mu} }\frac{1}{\ell^{1/2}},
\end{align*}
where the last inequality follows from Lemma \ref{L:sizeY2} and the fact that $\ell\in \mathcal{L}|_{\mu}$ if and only if  $\mu\in \mathcal{L}|_{\ell}$.
Now by Lemma \ref{lem:size-mu} and the fact that $m_1\geq 1$ (recall that
 $m_1=\widetilde{\E}^{(1),\infty}[\tau_1]=\widetilde{\E}^{(2),\infty}[\tau_1]$), we have
for $\mu$ large enough that
\begin{align*}
\sum_{\ell\in [c_1,\,\mu]\cap \mathcal{L}|_{\mu} }\frac{1}{\ell^{1/2}}
&\geq  
\sum_{\ell\in [c_1,\,\mu]\cap \left[\frac{\mu}{m_1} - c_5 \sqrt{\mu},\; \frac{\mu}{m_1} + c_5 \sqrt{\mu} \right] }\frac{1}{\ell^{1/2}}\\
&=  \sum_{\ell\in  \left[\frac{\mu}{m_1} - c_5 \sqrt{\mu},\; \frac{\mu}{m_1} + c_5 \sqrt{\mu} \right] }\frac{1}{\ell^{1/2}}\\
&\geq 
\frac{2 c_5\sqrt{\mu} -1}{\sqrt{\frac{\mu}{m_1} + c_5 \sqrt{\mu}}},
\end{align*}
which tends to $2c_5\sqrt{m_1}>0$ as $\mu\to\infty$.
We finally obtain \eqref{eq:pointwise-ztilde2} by combining the last display with \eqref{FinalBound1}.
\end{proof}

\section{Concluding remarks}\label{S:Conclude}


{\color{b}
Our main result, Theorem \ref{T:main}, suggests that to develop statistically consistent $k$-mer-based methods under a standard model such as the CFN model, one cannot simply fix a $k$ and use the $k$-mer frequencies of the entire leaf sequences.  Instead, one has to look for more elaborate methods, such as block decomposition.

Another possible approach to achieve consistency is to allow $k$ to increase with the sequence length $m$. It is an interesting open problem to determine the \textit{smallest} growth rate of $k=k_m$ as a function of $m$ for which consistency becomes possible (without recourse to block techniques). 
By standard phylogenetic reconstruction results for distance-based and likelihood-based methods (see, e.g.,~\cite{chang_full_1996,roychoudhury_consistency_2015}), statistical consistency is possible in the extreme case where $k_m=m$ (i.e., when the full sequence is observed). Formally, by the reconstruction upper bound \cite[Lemma 3.2]{fan2018}, it follows that
\begin{equation}\label{keqm}
\limsup_{m \to\infty }
\|\mathcal{L}^{(1)}_m - \mathcal{L}^{(2)}_m\|_\tv = 1 \quad\text{if }k_m=m.
\end{equation}
Hence the remaining question is:
Is there a sequence $\{k_m\}$ such that,
say, $\lim_{m\to\infty}\frac{k_m}{m} < 1$ or even $\lim_{m\to\infty}\frac{k_m}{m} = 0$, and
such that we also have $
\limsup_{m \to\infty }
\|\mathcal{L}^{(1)}_m - \mathcal{L}^{(2)}_m\|_\tv = 1 ?$
A key challenge to extend our proof 
is that, in our use of the local 
CLT, we must also control the convergence rate in terms of the dimension $d=1+3(2^{2k}-2^k)$. 

We focused exclusively on the two-state symmetric model of single-site substitution. We conjecture that the techniques developed here can be used to analyze more complex substitution models as well (e.g., the four-state Jukes-Cantor model).
Another open question is to show that our results hold under models of insertions and deletions, such as the TKF91 model~\cite{Thorne1991}. See~\cite{fan2020impossibility} for related results regarding estimators based on the sequence length alone.

}



\section*{Acknowledgments}

SR was supported by NSF grants DMS-1614242, CCF-1740707 (TRIPODS), DMS-1902892, DMS-1916378 and MS-2023239 (TRIPODS Phase II), as well as a Simons Fellowship and a Vilas Associates Award. BL was supported by NSF grants DMS-1614242, CCF-1740707 (TRIPODS), DMS-1902892 (to SR). WTF was supported by NSF grants DMS-1614242 (to SR), DMS-1855417 and ONR grant TCRI N00014-20-1-2411.

\bibliographystyle{alpha}
\bibliography{lengthbib}

\appendix

\section{Information-theoretic bounds}
\label{sec:app-information}


In this section we give details about some basic facts we used in the paper. Recall  the definition of the total variation distance in \eqref{Def_TV}.
It is well known, see e.g., \cite{LevinPeresWilmer2006}, that the supremum on the right hand side of \eqref{Def_TV} is reached at the set $B=\{x\in E:\,\nu_1(x)\geq \nu_2(x)\}$ as well as its complement $B^c$, and that we have the following characterizations.
\begin{lemma}
\label{lem:tv-pointwise}
Let $\nu_1$ and $\nu_2$ be probability measures on a countable space $E$.
\begin{align*}
\|\nu_1 - \nu_2\|_\tv
=\frac{1}{2}\sum_{\sigma \in E}
\left|\nu_1(\sigma)-\nu_2(\sigma)\right| 
= 1 - \sum_{\sigma \in E}\nu_1(\sigma)\land\nu_2(\sigma). \end{align*}
\end{lemma}


Let $X$ be a measurable function on a measure space $(\Omega,\,\mathcal{F})$, and $\P$ and $\P'$ be two probability measures on  $(\Omega,\,\mathcal{F})$. Denote by $\P_{g(X)}$ and $\P'_{g(X)}$ the probability distribution of $g(X)$ under $\P$ and $\P'$ respectively, where $g$ is an arbitrary measurable function on the state space of $X$.
\begin{lemma}
\label{lem:tv-f}
Let $g$ be a measurable map on the state space of $X$. Then 
$$
\|\P_{g(X)} - \P'_{g(X)}\|_\tv
\leq \|\P_{X} - \P'_{X}\|_\tv.
$$
\end{lemma}

\begin{proof}
Applying the definition \eqref{Def_TV} twice,
\begin{align*}
\|\P_{g(X)} - \P'_{g(X)}\|_\tv =& \sup_{A}|\P(g(X)\in A)-\P'(g(X)\in A) |\\
= & \sup_{A}|\P(X\in g^{-1}(A))-\P'(X\in g^{-1}(A)) |\\ 
& \leq \|\P_{X} - \P'_{X}\|_\tv.
\end{align*}
\end{proof}

Let $X,Y,Z$ be measurable functions on a measure space $(\Omega,\,\mathcal{F})$, and $\P$ and $\P'$ be two probability measures on  $(\Omega,\,\mathcal{F})$. We say that  $X \to Y \to Z$ is a \textit{Markov chain} under $\P$ if $Z$ is conditionally independent of $X$ given $Y$ in the sense that 
\begin{equation}\label{Markov1}
\P_{Z|X,Y} = \P_{Z|Y},    
\end{equation}
where $\P_{Z|X,Y}$ is the conditional distribution of $Z$ given $(X,Y)$ and
$\P_{Z|Y}$ is the conditional distribution of $Z$ given $Y$. The  law of total probability and \eqref{Markov1} imply that
\begin{equation}\label{Markov2}
    \P_{X,Y,Z} = \P_{X}\P_{Y|X}\P_{Z|Y},
\end{equation}
where $\P_{X,Y,Z}$ is the joint probability distribution of $(X,Y,Z)$.

\begin{lemma}
\label{lem:tv-markov}
Suppose  $\P_{X} = \P_{X}'$, $\P_{Y|X} = \P_{Y|X}'$ and $X \to Y \to Z$ is a Markov chain under both $\P$ and $\P'$. Then
$$
\|\P_{X,Y,Z} - \P'_{X, Y, Z}\|_\tv
= \|\P_{Y,Z} - \P'_{Y, Z}\|_\tv.
$$
\end{lemma}

\begin{proof}
By the first equality in Lemma \ref{lem:tv-pointwise},
\begin{align*}
    &\|\P_{X,Y,Z} - \P'_{X,Y,Z}\|_{\tv}\\
    &= \frac{1}{2}\sum_{(a,b,c)}\big|\P((X,Y,Z)=(a,b,c)) - \P'((X,Y,Z)=(a,b,c))\big|.
\end{align*} 
Applying \eqref{Markov2} to $\P$ and $\P'$, we have 
\begin{align*}
\P((X,Y,Z)=(a,b,c)) & =\P(X =a)\,\P(Y =b|X =a)\,\P(Z =c|Y =b),\\
\P'((X,Y,Z)=(a,b,c)) & =\P'(X =a)\,\P'(Y =b|X =a)\,\P'(Z =c|Y =b).
\end{align*} 

From the assumptions $\P_{X} = \P'_{X}$ and $\P_{Y|X} = \P'_{Y|X}$, it follows that $\P_{X,Y} = \P'_{X,Y}$ and $\P_{Y} = \P_{Y}'$. {\color{b}Using the displayed equations above gives} 
\begin{align*}
    &\left|\P((X,Y,Z)=(a,b,c)) - \P'((X,Y,Z)=(a,b,c)) \right| \\
     &= \,\P(X =a)\,\P(Y =b|X =a)\;\left|\P(Z =c|Y =b) - \P'(Z =c|Y =b)\right| \\
     &= \,\P(X=a, Y =b)\;\left|\P(Z =c|Y =b) - \P'(Z =c|Y =b)\right|.
\end{align*} 
Hence
\begin{align*}
    &\|\P_{X,Y,Z} - \P'_{X,Y,Z}\|_{\tv}\\ 
    &= \frac{1}{2}\sum_{(a,b,c)}\P(X=a, Y =b)\;\big|\P(Z =c|Y =b) - \P'(Z =c|Y =b)\big|\\
 & = \frac{1}{2}\sum_{(b,c)}\P(Y =b)\;\big|\P(Z =c|Y =b) - \P'(Z =c|Y =b)\big|\\  
  & = \frac{1}{2}\sum_{(b,c)}\big|\P(Y =b)\,\P(Z =c|Y =b) - \P'(Y =b)\,\P'(Z =c|Y =b)\big|,
\end{align*} 
where we used $\P(Y =b)=\P'(Y =b)$ in the last equality. {\color{b}The  expression on the last line is equal to $\|\P_{Y,Z} - \P'_{Y,Z}\|_{\tv}$} by the first equality in Lemma \ref{lem:tv-pointwise}, establishing the claim. 
\end{proof}



\end{document}